\documentclass[a4paper]{article}
\pdfoutput=1

\usepackage{fullpage}
\usepackage[utf8]{inputenc}
\usepackage[T1]{fontenc}
\usepackage{lmodern}
\usepackage{amsthm,amsmath,amsfonts,amssymb,mathtools}
\usepackage{graphicx}
\usepackage{microtype}
\usepackage{subcaption}
\usepackage[hidelinks]{hyperref}
\usepackage{stmaryrd}
\usepackage{bm}

\newtheorem{theorem}{Theorem}[section]
\newtheorem{corollary}[theorem]{Corollary}
\newtheorem{lemma}[theorem]{Lemma}
\newtheorem{proposition}[theorem]{Proposition}

\newcommand{\K}{\ensuremath{{\mathbb K}}}
\newcommand{\N}{\ensuremath{{\mathbb N}}}
\newcommand{\R}{\ensuremath{{\mathbb R}}}
\newcommand{\C}{\ensuremath{{\mathbb C}}}

\newcommand{\bx}{{\mathbf{x}}}
\newcommand{\bu}{{\mathbf{u}}}

\newcommand{\bz}{{\mathbf{z}}}
\newcommand{\bb}{{\mathbf{b}}}

\newcommand{\hbA}{{\hat{\bA}}}
\newcommand{\bV}{{\mathcal{V}}}
\newcommand{\bA}{{\mathbf{A}}}
\newcommand{\hA}{{\hat{\bA}}}
\newcommand{\barA}{{\bar{\bA}}}
\newcommand{\bB}{{\mathbf{B}}}
\newcommand{\bC}{{\mathbf{C}}}
\newcommand{\bI}{{\mathbf{Id}}}
\newcommand{\bM}{{\mathbf{M}}}
\newcommand{\bG}{{\mathbf{G}}}
\newcommand{\bS}{{\mathbf{S}}}

\newcommand{\abs}[1]{{\left\lvert #1 \right\rvert}}
\newcommand{\floor}[1]{{\left\lfloor #1 \right\rfloor}}
\newcommand{\ceil}[1]{{\left\lceil #1 \right\rceil}}
\newcommand{\RI}{\mathrm{RI}}
\newcommand{\core}[1]{\left\llbracket\, \begin{matrix} #1 \end{matrix} \,\right\rrbracket}

\DeclareMathOperator{\rank}{rk}
\DeclareMathOperator{\spann}{span}

\numberwithin{equation}{section}

\title{\textbf{On low-rank tensor train approximability for linear nearest neighbor systems}}

\author{Patrick Gel\ss{}\thanks{AI in Society, Science, and Technology, Zuse Institute Berlin, 14195 Berlin, Germany} \qquad Sebastian Matera\thanks{Theory Department, Fritz Haber Institute of the Max Planck Society, 14195 Berlin, Germany} \qquad Reinhold Schneider\thanks{Institute of Mathematics, Technical University of Berlin, 10623 Berlin, Germany} \qquad Andr\'e Uschmajew\thanks{Institute of Mathematics \& Centre for Advanced Analytics and Predictive Sciences, University of Augsburg, 86159 Augsburg, Germany}}

\date{}

\begin{document}

\maketitle

\begin{abstract}
Low-rank tensor methods are an important tool in the numerical treatment of equations with a high-dimensional state space. Nearest neighbor interaction systems like the Ising model or more general Markov jump processes, as well as 1D finite-state quantum systems are examples of such problems. While low-rank tensor train/matrix product state models have been shown to be highly efficient for the simulation of such systems, providing theoretical justification for this remains a challenging task. One approach for obtaining estimates on required ranks for certain accuracies is to investigate the rank increase in Krylov subspace methods for solving the problem at hand. In the context of area laws for ground states of 1D spin systems, nontrivial results on rank-increasing properties of nearest neighbor operator polynomials have been obtained in work of Arad et al.~[arXiv:1301.1162] by studying the partial commutativity of local operators. In the present work, this technique is applied to polynomial methods for definite linear equations and dissipative linear ODEs with nearest neighbor structure. This allows to derive corresponding low-rank approximability statements for solutions of such problems which are independent of the system size. Numerical simulations of high-dimensional nearest neighbor systems illustrate the theoretical findings.
\end{abstract}

\section{Introduction}

We consider low-rank approximation to solutions of high-dimensional linear equations
\begin{equation}\label{eq: linear system}
\bA \bu = \bb
\end{equation}
or linear ordinary differential equations
\begin{equation}\label{eq: linear ODE}
\frac{d}{dt} \bu = \bA \bu, \quad \bu(0) = \bu_0,
\end{equation} 
that are posed on a $d$-fold tensor product space,
\begin{equation}\label{eq: tensor product space}
\bV \coloneqq \bigotimes_{\mu=1}^{d} {V}_{\mu},
\end{equation}
where $V_{\mu}$ are finite-dimensional $\mathbb{K}$-vector spaces ($\mathbb{K} = \R$ or $\mathbb{K} = \C$) of dimension $n_\mu \ge 2$. Here $\bA$ is a linear operator on $\bV$. Such equations in tensor product spaces arise in several situations, one being the discretization of linear partial differential or integral equations in spaces of multivariate functions. This work, however, is mainly motivated by applications in inverse problems~\cite{HOLTZ2012, GELSS2019a, GELSS2019b}, Markov jump processes~\cite{KAZEEV2014, DOLGOV2015, GELSS2016, GELSS2017}, and finite-state quantum systems~\cite{ORUS2019, GELSS2025, SANDER2025}.

By identifying the tensor product space $\bV$ with the space $\K^{n_1 \times \dots \times n_d}$ via a fixed tensor product basis, a tensor $\bu \in \mathcal{V}$ is identified with a $d$-dimensional (in case of~\eqref{eq: linear ODE} time-dependent) array
\[
\bu \cong [\bu(i_1,\ldots,i_d)]
\]
of coefficients in $\mathbb K$ indexed via $d$ discrete indices $i_\mu \in \{1, \ldots , n_\mu\}$. Since the dimension $n_1 \cdots n_d$ of the space $\bV$ grows exponentially with $d$, the practical representation of its elements, and even more so, the numerical solution of equations like~\eqref{eq: linear system} or~\eqref{eq: linear ODE} defined on that space poses great challenges for large $d$. This is often referred to as the \emph{curse of dimensionality}. It is therefore of interest to identify problems which nevertheless allow for efficient numerical treatment because their solution can be approximated using low-parametric (data-sparse) representations.

In the last decades, low-rank tensor techniques have been developed as a powerful tool to deal with high-dimensional problems under suitable structural assumptions on operators and data. The idea of these methods as outlined in seminal works such as~\cite{Beylkin2002} is to apply low-parametric representations of higher-order tensors based on suitable \emph{low-rank tensor formats}. Foundational mathematical aspects are presented in the monographs~\cite{Hackbusch2012,Khoromskij18}, while comprehensive overview on different tensor formats, algorithms and applications is given in survey articles~\cite{GrasedyckKressnerTobler2013,Hackbusch2014,Szalay15,Bachmayr16,UV20,Bachmayr2023}. However, while low-rank tensor formats are routinely and successfully used for numerical computations in practice, a rigorous approximation theory remains difficult to establish for them due to the nonlinearity of the underlying approximation model. This means that for most problems no strong a-priori statement can be made as to whether or not its solution will be well-approximable in a certain low-rank tensor format.

In this work, we are interested in the low-rank approximability of solutions to linear problems~\eqref{eq: linear system} and~\eqref{eq: linear ODE} when using a particular tensor format, the so-called \emph{tensor train} or \emph{TT format}~\cite{Oseledets2011}. In the TT format, tensors (more precisely: their coefficient arrays $\bu \in \K^{n_1 \times \dots \times n_d}$ with respect to a fixed tensor product basis) are represented entry-wise as
\begin{equation}\label{eq: TT format}
\bu(i_1,\dots,i_d) = G_1(i_1) \cdot G_2(i_2) \cdots G_d(i_d),
\end{equation}
with certain matrices $G_\mu(i_\mu) \in \mathbb{K}^{r_{\mu-1} \times r_\mu}$ for $i_\mu = 1,\dots,n_\mu$, where $r_0 = r_d = 1$. In physics, tensors of this form are called {\em matrix product states (MPS)}~\cite{white1992,Schollwoeck2011}. Such a representation of a tensor is ``low-parametric'' if the values $r_1,\dots,r_{d-1}$ that determine the sizes of the matrices $G_\mu(i_\mu)$ are sufficiently small. Indeed, assuming $r_\mu \le r$ for all $\mu$, storing all those matrices requires storing not more than $dnr^2$ values (instead of the $n_1 \cdots n_d$ entries of $\bu$). In particular, the dependence on $d$ is only linear and the curse of dimensionality is broken in cases when $r$ does not grow square-root exponentially with $d$.

The required values of $r_1,\dots,r_{d-1}$ for exactly representing a given tensor $\bu$ in the TT format~\eqref{eq: TT format} can be characterized algebraically as follows. For every $\mu \in \{1,\dots,d-1\}$ we may interpret $\bu \in \bV$ as a second-order tensor (that is, a matrix)
\[
\bu^\mu \in \left(\bigotimes_{\nu \le \mu} V_{\nu} \right)  \otimes \left(\bigotimes_{\nu > \mu} V_{\nu} \right) \eqqcolon \bV_{\le \mu} \otimes \bV_{> \mu}.
\]
Identifying $\bu$ as an array in $\K^{n_1 \times \dots \times n_d}$ as before, $\bu^\mu$ is simply obtained by reshaping this array into a matrix of size $(n_1\cdots n_\mu) \times (n_{\mu+1} \cdots n_d)$. The (matrix) ranks of $\bu^\mu$ for $\mu=1,\dots,d-1$ are called the \emph{$\mu$-ranks} or \emph{TT ranks} of $\bu$ and will be denoted by
\[
\rank_\mu(\bu) \coloneqq \mathrm{rank}(\bu^\mu), \qquad \mu=1,\dots,d-1.
\]
It then can be shown~\cite{Oseledets2011} that $\bu$ is \emph{exactly} representable in the TT format~\eqref{eq: TT format} for some matrices $G_\mu(i_\mu)$ of size $r_{\mu-1} \times r_\mu$ if and only if $r_\mu \ge \rank_\mu(\bu)$ for $\mu=1,\dots,d-1$. In particular, the most efficient representation would use $r_\mu = \rank_\mu(\bu)$. On the other hand, we clearly have
\begin{equation}\label{eq: trivial estimate}
\rank_\mu(\bu) \le \min (n_1\cdots n_\mu, n_{\mu+1} \cdots n_d)
\end{equation}
and a ``random'' tensor $\bu$ will actually satisfy equality for $\mu=1,\dots,d-1$. Therefore, an exact low-parametric representation in TT format cannot be expected. However, a tensor might still be \emph{approximable} by a tensor with small TT ranks. This is in fact the case if and only if the singular values of all the matricizations $\bu^\mu$ decay sufficiently fast and the so called TT-SVD algorithm even provides a constructive procedure for obtaining nearly optimal approximations with low TT ranks~\cite{Oseledets2011}.

Proving in advance (a priori) that the solution $\bu_* \in \bV$ of some high-dimensional equation, say of the form~\eqref{eq: linear system} or~\eqref{eq: linear ODE}, will have the property of fast decaying singular values, and in turn can be well approximated in TT format with small ranks, is a nontrivial task. From a perspective of approximation theory, this could be achieved by providing decay estimates for relative ``low-rank approximation numbers''
\begin{equation}\label{eq: approximation numbers}
\tau^{(r)}(\bu_*) \coloneqq \min_{\max_\mu \rank_\mu(\bu) \le r} \frac{\| \bu_* - \bu \|}{\| \bu^* \|}
\end{equation}
in dependence of $r$. Here and in the following, $\| \cdot \|$ denotes a Euclidean norm generated by a canonical inner product $\langle \cdot, \cdot \rangle$ on $\bV = V_1 \otimes \dots \otimes V_d$, that is, an inner product which is inherited from some given inner products on $V_1,\dots,V_d$ via $\langle u_1 \otimes \dots \otimes u_d, v_1 \otimes \dots \otimes v_d \rangle = \langle u_1, v_1 \rangle_{V_1} \cdots \langle u_d, v_d \rangle_{V_d}$. It can be shown that for any fixed $r$ the set of all tensors satisfying $\max_\mu \rank(\bu) \le r$ is closed, so the number $\tau^{(r)}(\bu_*)$ is well-defined (the minimum is achieved) for any $\bu_* \in \bV$ and $r \in \mathbb N$. Note that for linear equations~\eqref{eq: linear system}, instead of~\eqref{eq: approximation numbers} we will later in fact consider corresponding approximation errors in $\bA$-norm (for Hermitian positive definite $\bA$) or residual norm, as defined in equations~\eqref{eq: tau A norm} and~\eqref{eq: tau residual norm}.

When $\bu_*$ corresponds to the discretization of a multivariate function (e.g.~in a PDE context), one possible approach for estimating the decay of $\tau^{(r)}(\bu_*)$ with respect to $r$ are a-priori regularity estimates either based on classical smoothness or Sobolev regularity. These imply certain decay of singular values of matricizations of~$\bu_*$~\cite{Temlyakov1992,Griebel2019,Griebel2023}. However, such a generic approach that only focuses on the regularity of the solution may lead to weak approximability results in cases where specific low-rank promoting structure is present in the equation, notably in the operator $\bA$. In addition, considerations based on regularity hardly apply in other, intrinsically finite-dimensional contexts such as the simulation of multi-state reaction networks that motivated the present work. From an algorithmic approach, constructive low-rank approximation schemes such as TT-cross approximation \cite{OSELEDETS2010} are by now standard tools, but a detailed error analysis is available only in certain settings~\cite{QinEtAl2022}. These results typically rely on specific access models or structural assumptions on the target tensor that can be difficult to check beforehand. Also, they do not directly address how the dynamics generated by a given operator $\bA$ with certain structure affect the decay of the approximation errors~\eqref{eq: approximation numbers} for the solution.

Specifically, it is commonly assumed that a tensor product structure of the operator $\bA$ together with low-rank problem data (i.e.~the right hand side $\bb$ in~\eqref{eq: linear system} and the initial state $\bu_0$ in~\eqref{eq: linear ODE}) should imply some kind of low-rank approximability of the solution $\bu_*$ for somewhat more algebraic reasons. This can be verified in some specific cases. For example, in the theory of linear matrix equations it is well-known that solutions of Sylvester equations with low-rank right hand side and symmetric and (positive) definite system matrices admit fast decaying singular values; see, e.g.,~\cite[Sec.~4.4]{Simoncini2016} for reference. This result can be generalized to linear tensor equations of a similar structure, where the linear operator $\bA$ is a sum of symmetric (positive) definite operators, each acting only on single modes of the tensor. Such operators are often referred to as \emph{Laplace type} operators in the tensor literature, because this structure arises in standard discretizations of the Laplace operator on an axis aligned grid. Due to the commutativity of summands, the exponential $\exp(t\bA)$ of such a Laplace type operator turns out to be a simple Kronecker product which implies that the solution of an ODE~\eqref{eq: linear ODE} with such an operator will be of constant rank for all $t$, except for degenerate cases. This observation does not even require symmetry or definiteness of the single-mode operators. On the other hand, when $\bA$ is a symmetric and (positive) definite Laplace type operator with a sufficiently large smallest eigenvalue, then the inverse $\bA^{-1}$ can be extremely well approximated by sums of exponentials of $\bA$, that is, by a short sum of Kronecker products. This does not only provide a versatile preconditioner in various scenarios, but also indicates a constructive proof to the fact that solutions of linear equation~\eqref{eq: linear system} with Laplace type operators and low-rank right-hand side can be well approximated using low-rank tensors even in extremely high dimension~\cite{Beylkin2002,Grasedyck2004}. For more details on approximation by exponential sums see~\cite[Sec.~9.8.2]{Hackbusch2012}.

This work follows a similar constructive approach for deriving decay estimates for the quantities~$\tau^{(r)}$ for solutions of linear problems~\eqref{eq: linear system} or~\eqref{eq: linear ODE}, but with operators $\bA$ of \emph{nearest neighbor type}, that is, sums of linear operators~$\bA_{\mu}$ which only act on two neighboring modes $\mu$ and $\mu+1$ of the tensor, see~\eqref{eq: nearest neighbor operator} for the precise definition. Such operators arise in many applications. However, instead of approximation by exponentials as for Laplace type operators, we will rely on the somewhat more basic idea of polynomial approximation. This means that the solution $\bu_*$ to one of the equations of interest is approximated by
\begin{equation}\label{eq: polynomial approximation}
\bu_* \approx \bu_\ell = p_\ell(\bA) \hat \bu,
\end{equation}
where $p_\ell$ is a suitable polynomial of degree $\ell$, and $\hat \bu$ either equals $\bb$ in case of linear equations~\eqref{eq: linear system} or $\bu_0$ in case of ODEs~\eqref{eq: linear ODE}. This is a classic and fundamental idea in numerical linear algebra and analysis. Note that $\bu_\ell$ given by~\eqref{eq: polynomial approximation} is in fact an element of the so-called Krylov subspace
\begin{equation}\label{eq: Krylov subspace}
\mathcal{K}_{\ell+1} = \spann\{ \hat \bu, \bA \hat \bu, \bA^2 \hat \bu ,\dots, \bA^\ell \hat \bu \},
\end{equation}
and practical methods realizing (usually implicitly) approximations of the form~\eqref{eq: polynomial approximation} are referred to as Krylov subspace methods. For many such methods, approximation errors and convergence rates in terms of polynomial degree $\ell$ are well understood based on spectral properties of $\bA$. Having estimates on the $\mu$-ranks of $\bu_\ell$ would hence naturally leads to ``constructive'' estimates for quantities like~$\tau^{(r)}(\bu_*)$ and hence statements on low-rank approximability. This suggests to inspect how the nearest neighbor structure of $\bA$ affects rank-increasing properties of polynomials $p_\ell(\bA)$ when applied to low-rank data~$\hat \bu$.

Such an approach can be motivated from two sides. First, nearest neighbor operators are arguably the simplest non-trivial generalization of Laplace type operators. As the latter, nearest neighbor-operators still allow efficient computations in low-rank TT format even in high dimensions (as for instance, showcased by our numerical experiments). However, in contrast to Laplace type operators, a theoretical explanation for low-rank approximability of solutions is not as straightforward due to a (partial) loss of commutativity. Notably, applying the exponential function to such operators does not yield simple Kronecker products. Still, one would expect that impact of this effect can be analyzed as it only affects ``neighboring'' operators. Using polynomials instead of exponential sums seems more amenable for this task.

The second motivation comes from physics. Matrix product states (the physics terminology for the TT format) have been developed for the simulation of one-dimensional quantum spin systems~\cite{white1992,Schollwoeck2011}, which have a natural nearest neighbor structure. This is clearly reflected in the linear structure of the TT format~\eqref{eq: TT format} as it fixes a certain ordering of the indices $i_1,\dots,i_d$ and only neighboring indices interact directly in the matrix multiplication. By physical considerations, when initialized with separable states, the entanglement caused by local interaction spreads only slowly through the system, implying a certain low-rank approximability of ground states. The famous work of Hastings~\cite{Hastings2007} on one-dimensional area laws for the entanglement entropy provided a first proof for this fact. Our work is inspired by a later approach proposed by Arad and collaborators in a series of works~\cite{Aharonov2011,Aradetal2012,Aradetal2013,Arad2015}, which is based on polynomial approximation of the ground state projector and analysis of its rank-increasing properties. One of the remarkable results from~\cite{Aradetal2013}, as reproduced in Theorem~\ref{th: result by Arad et al} and Corollary~\ref{cor: corollary on rank increase} below, states that applying a degree $\ell$ operator polynomial of a nearest neighbor operator to any tensor will increase its TT ranks by a factor which grows only subexponentially in $\ell$. As a consequence, if the convergence of $\bu_\ell$ to $\bu_*$ in~\eqref{eq: polynomial approximation} is exponentially fast, as is the case for many polynomial methods, one can deduce nontrivial estimates for the TT ranks required for a certain accuracy. In the present paper, we follow this logic for the case of linear equations~\eqref{eq: linear system} and ODEs~\eqref{eq: linear ODE} with nearest neighbor operators.

While not directly related to our discussion on tensors given as solutions to equations, we can mention the recent work~\cite{RohrbachEtAl2022} as an example where low-rank approximability of Gaussian densities is considered using a combination of polynomial approximation and matrix decomposition.

\subsection*{Contribution and Outline}

Our main goal is to obtain decay estimates for the quantities~\eqref{eq: approximation numbers} by combine well-known convergence results for Krylov subspace methods for linear equations and ODEs with the estimates on rank-increasing properties of nearest neighbor operator polynomials provided in~\cite{Aradetal2013}. While the authors of~\cite{Aradetal2013} did this for Hermitian eigenvalue problems, we believe that working out the details for linear equations and linear ODEs is equally insightful and to our knowledge has not been presented before. For linear equations a similar pattern of comparing the convergence of polynomial approximation with rank growth has been applied in~\cite{Kressner2016}, but here we are able to use the non-trivial results from~\cite{Aradetal2013} for the case of nearest neighbor systems, which we acknowledge as the main technical tool of the present work.

In Section~\ref{sec: rank-increase} we first present the ideas from~\cite{Aradetal2013} on rank-increasing properties of nearest neighbor operator polynomials and state their main result (although in modified and more explicit form) as Theorem~\ref{th: result by Arad et al}. In addition, we include a detailed proof in the appendix. In Section~\ref{sec: implications} we then combine the rank-increasing properties of operator polynomials with available convergence results of some standard Krylov subspace methods for definite linear equations (Section~\ref{sec: linear equations}) and linear dissipative ODEs (Section~\ref{sec: odes}) to obtain statements on required TT ranks for approximating the solution with a certain accuracy. This implies decay rates for quantities $\tau^{(r)}(\bu_*)$ with respect to growing $r$. Importantly, these rates are independent of the tensor order $d$, that is, the system size, under suitable assumptions. For both linear equations and ODEs we state results that cover the non-Hermitian (but still definite/stable) case as well. The main results are Theorem~\ref{th: main theorem for linear systems} (definite Hermitian linear equations), Theorem~\ref{th: theorem for non-hermitian linear systems} (definite non-Hermitian linear equations), and Theorem~\ref{thm: main result ODEs} (linear dissipative ODEs). Finally, in Section~\ref{sec: applications} we present results of extensive numerical experiments that demonstrate the effective use of low-rank TT tensors for nearest neighbor systems, and confirm the theoretical error versus rank bounds from Section~\ref{sec: implications}.

We conclude this introduction with several comments on the scope and limitations of our results: First, our results in the current form are restricted to finite-dimensional spaces $V_\mu$, notably to a uniform bound $n_\mu \le n$ on the dimensions. This is due to an algebraic trick in the proof of Theorem~\ref{th: result by Arad et al} which leads to an appearance of the term $n^{\sqrt{\ell}}$ in the constant of~\eqref{eq: simplified rank bound}. Taking the limit $n \to \infty$ will therefore lead to infinite constants. Consequently, the tensor space $\bV$ also needs to have finite dimension ${}\lesssim n^d$. Nevertheless, some of our statements on the decay of quantities like $\tau^{(r)}(\bu_*)$ will be deduced in an ``asymptotic'' way (taking $r$ arbitrary large) for convenience. Note, that our results remain non-trivial for arbitrary large $d$ though, since the rates for $\tau^{(r)}(\bu_*)$ are independent from $d$.

Second, the whole approach of polynomial approximation of inverses or exponentials is limited to operators with bounded or controllable condition number. Consequently, our results might be less useful when applied to tensor-structured discretizations of high-dimensional PDEs. There are, however, other important applications where our assumptions are realistic such as nearest neighbor dynamics of spin systems or master equations for high-dimensional probabilities.

Third, the use of polynomial approximations and Krylov subspace methods in this work mainly serves a theoretical purpose for proving low-rank approximability of exact solutions. We do not propose or advocate any particular practical method for numerically solving high-dimensional linear equations or ODEs. While Krylov-type methods have been successfully applied for several types of tensor-structured equations~\cite{KressnerTobler2011, Chen2012, Ballani2013, Dolgov2013, Beik2016, Bentbib2020, Bucci2025, Casulli2026}, other approaches such as alternating optimization methods for linear equations~\cite{HOLTZ2012,DOLGOV2014} (in fact, also used in our experiments), or dynamical low-rank approximation~\cite{LubichOV2015,Haegeman2016,Conte2020} for ODEs, can be more efficient as they directly exploit the multilinear representation of TT tensors. A recent performance comparison of linear solvers in TT format can be found in~\cite{RoehrigZoellner2025}.

Lastly, we explicitly mention that we do not include results on eigenvalue problems in this work. As already stated, the case of (Hermitian) eigenvalue problems is well studied in physics in the context of area laws and was in particular treated in original works like~\cite{Aradetal2012,Aradetal2013} by Arad and coworkers. Compared to linear equations and ODEs, eigenvalue problems come with additional difficulties. Let us comment on one theoretical aspect already discussed in~\cite{Kressner2016}. In principle, as done in~\cite{Aradetal2012}, we can use polynomial approximations~\eqref{eq: polynomial approximation} for approximating an eigenvector $\bu_*$ of an Hermitian operator $\bA$ by using a vector~$\hat \bu$ that is not orthogonal to the corresponding eigenspace of $\bu^*$, and polynomial approximations $p_\ell(\bA)$ of the orthogonal eigenspace projector, that is, the polynomial $p_\ell$ should approximate the characteristic function of the eigenvalue within the spectrum. The convergence rate with respect to $\ell$, say when using suitably rescaled Chebyshev polynomials, will be exponential, but depend on the relative spectral gap of the target eigenvalue~\cite[Lemma~4.1]{Aradetal2013}. When $\bA$ is a nearest neighbor operator, the rank-increasing properties of $p_\ell(\bA)$ apply as in the other cases, which seems to indicate similar conclusions as for linear equations or linear ODEs. However, there is one main difference: for eigenvalue problems, the vector $\hat \bu$ in~\eqref{eq: polynomial approximation} is not part of the problem data and hence simply assuming it to be low-rank might not be sensible. Note that the relative error of the approximate eigenspace projection will be inversely proportional to the cosine of the angle between $\hat \bu$ and the eigenspace. One is therefore left with the task of showing that low-rank elements $\hat \bu$ with a sufficiently large initial overlap to the eigenvector $\bu_*$ exist, in particular, within our context, an overlap independent of $d$ or at least not subject to the curse of dimensionality. (For example, taking as $\hat \bu$ the best rank-one approximation of $\bu^*$ would not satisfy this, as its guaranteed overlap is only of order $n^{-(d-1)/2}$.) At closer inspection, this task appears almost as difficult as showing low-rank approximability of $\bu^*$ itself, and requires additional effort and assumptions such as in the results on frustration free spin systems in~\cite{Aradetal2012}. In any case, results on eigenvalue problems are of different conceptual flavor than the cases of linear equations and ODEs, and are therefore not considered in this work.

\section{Nearest neighbor interaction operators and rank increase}

In this section we recall some of the main ideas of~\cite{Aradetal2013} and apply them in the general context of polynomial approximation methods~\eqref{eq: polynomial approximation}. We continue with the notation from the introduction, notably the tensor product space $\bV$ defined in~\eqref{eq: tensor product space}.

\subsection{Nearest neighbor interaction operators}

A linear operator $\bA \vcentcolon \bV \to \bV$ is called a \emph{nearest neighbor interaction operator} if it takes the form
\begin{equation}\label{eq: nearest neighbor operator}
\bA = \sum_{\mu=1}^{d-1} \bA_\mu
\end{equation}
where each linear operator $\bA_\mu$ effectively acts on $V_\mu \otimes V_{\mu + 1}$ only. More precisely,
\begin{equation}\label{eq: local operator}
\bA_{\mu} = \begin{cases}
 \tilde{A}_1 \otimes \bI_{> 2}, & \quad \text{for $\mu=1$,} \\
 \bI_{< \mu} \otimes \tilde{A}_\mu \otimes \bI_{> \mu+1}, & \quad \text{for $\mu=2,\dots,d-2$,} \\
 \bI_{< d-1} \otimes \tilde{A}_{d-1}, &\quad \text{for $\mu = d-1$},
\end{cases}
\end{equation}
where $\tilde{A}_\mu$ are linear operators on $V_\mu \otimes V_{\mu+1}$, that is, can be represented by a matrix of size $n_\mu n_{\mu+1} \times n_\mu n_{\mu+1}$, and $\bI_{< \mu} = I \otimes \dots \otimes I$ ($\mu - 1$ factors) and $\bI_{> \mu+1} = I \otimes \dots \otimes I$ ($d - \mu -1)$ factors) are identities on $V_1 \otimes \dots \otimes V_{\mu-1}$ and $V_{\mu+2} \otimes \dots \otimes V_d$, respectively. Such operators occur for instance in the modeling of one-dimensional particle systems or reaction networks; see Section~\ref{sec: applications}.

A crucial observation in the study of nearest neighbor interaction operators is that for every fixed~$\mu$ operators $\bA_\lambda$ and $\bA_\nu$ with $\lambda < \mu$ and $\nu > \mu$ always commute. This follows easily from~\eqref{eq: local operator}. Consequently, the operators
\[
\bA_{< \mu} \coloneqq \sum_{\lambda < \mu} \bA_\lambda, \quad \text{and} \quad \bA_{> \mu} \coloneqq \sum_{\nu > \mu} \bA_\nu
\]
commute as well.

\subsection{Rank increase in polynomial approximations}\label{sec: rank-increase}

We are interested in estimates of the $\mu$-ranks of polynomial approximations
\[
\bu_\ell = p_\ell(\bA) \hat \bu
\]
where $\bA$ is a nearest neighbor interaction operator~\eqref{eq: nearest neighbor operator}, $p_\ell$ is a polynomial of degree $\ell$ and $\hat \bu \in \bV$ represents \emph{low-rank data} of the problem in the sense that $\hat \bu$ is assumed to have small $\mu$-ranks. To obtain such estimates we can study how much the application of the operator polynomial $p_\ell(\bA)$ to $\hat \bu$ increases its $\mu$-ranks. In general, this is a quite nontrivial task and in addition can be rather problem dependent, that is, depend on additional assumptions. However, here we aim at ``generic'' estimates of the form
\begin{equation}\label{eq: goal estimate}
\max_{\mu} \rank_{\mu}(p_\ell(\bA) \bu) \le C \cdot \max_{\mu} \rank_\mu(\bu) 
\end{equation}
that hold \emph{for all} $\bu \in \bV$ and \emph{all} polynomials $p_\ell$ of degree $\ell$. The maxima are taken over $1 \le \mu \le d-1$. Such estimates are obtained by combining the rank-increasing properties of the local operators $\bA_\mu$ with their partial commutativity. The constant $C$ may depend on the dimensions $n_1,\dots,n_d$, the properties of local operators $\bA_\mu$, and the polynomial degree $\ell$, but should be independent of $\bu$. Specializing~\eqref{eq: goal estimate} to $\hat \bu$ will then give a result for $\bu_\ell$.

In the following discussion, we fix an ``inner'' index $2 \le \mu \le d-2$ (the boundary cases $\mu=1$ and $\mu=d$ require some obvious modifications and are actually simpler). Then we have
\begin{equation}\label{eq: decomposition of A}
\bA = \bA_{< \mu} + \bA_\mu + \bA_{> \mu}.
\end{equation}
Since, in light of~\eqref{eq: local operator}, it holds that $\bA_{< \mu} = \bB_{\le \mu} \otimes \bI_{>\mu}$ and $\bA_{> \mu} = \bI_{\le \mu} \otimes \bC_{> \mu}$ for some operators $\bB_{\le \mu}$ and $\bC_{> \mu}$ on $\bV_{\le \mu}$ and $\bV_{> \mu}$, respectively, the operators $\bA_{< \mu}$ and $\bA_{> \mu}$ each by themselves do not increase the $\mu$-rank when applied to any $\bu \in \bV$. Only the operator $\bA_\mu$ in the sum~\eqref{eq: decomposition of A} will increase the $\mu$-rank. The worst-case rank increase can be quantified by
\begin{equation}\label{eq: definition Rmu}
 \RI(\bA_\mu) \coloneqq \sup_{\bu \in \bV \setminus \{0 \}} \frac{\rank_\mu(\bA_\mu \bu)}{\rank_\mu(\bu)},
\end{equation}
which equals the smallest factor $R > 0$ such that
\[
\rank_\mu(\bA_\mu \bu) \le R \cdot \rank_\mu(\bu) \quad \text{for all $\bu \in \bV$.}
\]
By~\eqref{eq: decomposition of A} and sub-additivity of $\mu$-rank, we then have
\begin{equation}\label{eq: local rank increase}
\rank_\mu(\bA \bu) \le (\RI (\bA_\mu) + 2) \cdot \rank_\mu(\bu).
\end{equation}

An upper bound for $\RI(\bA_\mu)$ can be obtained from the Kronecker rank of its local representation $\tilde{A}_\mu$ on $\bV_\mu \otimes \bV_{\mu+1}$ as given in~\eqref{eq: local operator}. The Kronecker rank of $\tilde{A}_\mu$ is the smallest integer $R_\mu$ such that $\tilde{A}_\mu$ can be written as
\begin{equation*}\label{eq: low Kronecker rank decomposition of local operators}
 \tilde{A}_\mu = \sum_{q=1}^{R_\mu} A_{\mu,q} \otimes B_{\mu,q}.
\end{equation*}
with linear operators $A_{\mu,q}$ on $\bV_\mu$ and $B_{\mu,q}$ on $\bV_{\mu+1}$. Since $\tilde{A}_\mu$ can be regarded as an $(n_\mu n_{\mu+1}) \times (n_\mu n_{\mu+1})$ matrix, we always have 
\begin{equation}\label{eq: upper bound Kronecker rank}
\RI(\bA_\mu) \le R_\mu \le n_\mu n_{\mu+1}.
\end{equation}
It should be noted, however, that estimating $\RI(\bA_\mu)$ by the Kronecker rank of $\tilde A_\mu$ can be a too pessimistic. For example, when $\bV_\mu = \bV_{\mu+1}$ and $\tilde A_\mu$ is the transposition operator, then its Kronecker rank equals~$n_{\mu} n_{\mu+1}$, but $\RI(\bA_\mu) = 1$. For simplicity, in the following we will continue with a generic uniform upper bound
\begin{equation}\label{eq: generic RI A_mu estimate}
\RI(\bA_\mu) \le R
\end{equation}
for all $\mu$, keeping in mind that $R = \max_\mu R_\mu$ is a possible choice, but not necessarily the best one.

We next consider powers $\bA^\ell$ of $\bA$. Combining~\eqref{eq: local rank increase} and~\eqref{eq: generic RI A_mu estimate} in a naive way yields
\begin{equation}\label{eq: naive rank estimate}
\rank_\mu(\bA^\ell \bu) \le (R + 2)^\ell \cdot \rank_\mu(\bu)
\end{equation}
for any $\bu \in \bV$, which does not look very promising with regard to estimating the $\mu$-ranks of polynomial approximations $\bu_\ell = p_\ell(\bA) \hat \bu$. We can improve the estimate by taking into account that $\bA_{<\mu}$ and $\bA_{> \mu}$ commute. By expanding $\bA^\ell = (\bA_{< \mu} + \bA_\mu + \bA_{> \mu})^\ell$ we obtain $3^\ell$ monomials of order~$\ell$. Each monomial by itself increases the $\mu$-rank at most by the factor $R^k$, where $0 \le k \le \ell$ is the number of occurrences of the operator $\bA_\mu$ in it. However, thanks to commutativity of $\bA_{<\mu}$ and $\bA_{> \mu}$, many monomials are actually multiples of each other. For instance, when $\ell = 4$, we could subsume
\begin{equation}\label{eq: example commutativity}
\bA_{< \mu}^2 \bA_{> \mu}^{} \bA_\mu^{} + \bA_{< \mu}^{} \bA_{> \mu}^{} \bA_{< \mu}^{} \bA_\mu^{} + \bA_{> \mu}^{} \bA_{< \mu}^2 \bA_\mu^{} = 3 \bA_{< \mu}^2 \bA_{> \mu}^{} \bA_\mu^{},
\end{equation}
so this partial sum within the expansion of $(\bA_{< \mu} + \bA_\mu + \bA_{> \mu})^4$ increases the $\mu$-rank only by factor~$R$ (and not $4R$). In~\eqref{eq: example commutativity} we have used that $\bA_\mu$ appears at the same (the last) position within the three considered monomials, and that the number of occurrences of $\bA_{< \mu}$ is the same (two) in each monomial. In order to deal with the situation more systematically, we define for $0 \le k \le \ell$ (and fixed $2 \le \mu \le d-2$) the sets
\begin{equation}\label{eq: set Skl}
S_{k,\ell} = \{\bA_{<\mu}^{a_0} \bA_{> \mu}^{b_0} \bA_\mu^{} \bA_{<\mu}^{a_1} \bA_{>\mu}^{b_1} \bA_\mu  \cdots \bA_{<\mu}^{a_{k-1}} \bA_{>\mu}^{b_{k-1}} \bA_\mu^{} \bA_{<\mu}^{a_k} \bA_{>\mu}^{b_k} \vcentcolon a_i, b_i \in \N, \, \sum_i a_i + b_i = \ell - k \}
\end{equation}
of ``partially ordered'' monomials of order $\ell$ in which $\bA_\mu$ appears $k$ times. Since $\bA_{<\mu}$ and $\bA_{> \mu}$ commute, each monomial in the expansion of $\bA^\ell =(\bA_{< \mu} + \bA_\mu + \bA_{> \mu})^\ell$ matches exactly one such ``partially ordered'' monomial for some $k = 0,\dots,\ell$. We can therefore write
\begin{equation}\label{eq: expansion into monomials}
\bA^\ell = \sum_{k=0}^\ell \sum_{\bM \in S_{k,\ell}} K_\bM \cdot \bM
\end{equation}
for some coefficients $K_\bM \in \N$. Since each monomial $\bM$ in $S_{k,\ell}$ increases the $\mu$-rank at most by the factor~$R^k$, it follows that
\begin{equation}\label{eq: better estimate}
\rank_\mu(\bA^\ell \bu ) \le \left( \sum_{k=0}^\ell \abs{S_{k,\ell}} \cdot R^k \right) \cdot \rank_\mu(\bu),
\end{equation}
and it is an easy combinatorial task to find the value
\begin{equation}\label{abs Skl}
\abs{S_{k,\ell}} = \binom{\ell + k + 1}{2k+1}.
\end{equation}
Indeed, by inserting identities between $\bA_{<\mu}^{a_i}$ and $\bA_{> \mu}^{b_i}$ in~\eqref{eq: set Skl}, every element in $S_{k,\ell}$ is determined by the overall $2k+1$ positions of these identities and the $\bA_\mu$ within a monomial of total length $\ell + k+1$.

The new estimate~\eqref{eq: better estimate} is more precise than~\eqref{eq: naive rank estimate}, but the rank bound still grows exponentially in~$\ell$. For example, the term $R^\ell$ in~\eqref{eq: better estimate} arises solely from the monomial $\bM = \bA_\mu^\ell$ in~\eqref{eq: expansion into monomials}. However, as demonstrated by Arad et al.~in~\cite{Aradetal2013}, estimating the rank increase in the above way misses an important observation, namely that for example the said monomial $\bA_\mu^\ell$ does not increase the neighboring $(\mu+1)$-rank at all, that is, we have
\[
\rank_{\mu+1}(\bA_\mu^\ell \bu) \le \rank_{\mu+1}(\bu)
\]
for all $\bu \in \bV$. On the other hand, for algebraic reasons (cf.~Lemma~\ref{lem: shiFdting rank}) it always holds
\begin{equation*}
\rank_\mu(\bA_\mu^\ell \bu) \le n_{\mu+1} \cdot \rank_{\mu+1}(\bA_\mu^\ell \bu).
\end{equation*}
Therefore, we obtain the upper bound
\begin{equation}\label{eq: much better estimate}
\rank_\mu(\bA_\mu^\ell \bu) \le n_{\mu+1} \cdot \rank_{\mu+1}(\bu) \le \max_\mu n_\mu \rank_\mu(\bu)
\end{equation}
which does not grow exponentially with $\ell$. The same argument actually can be applied to all monomials~$\bM$ in~\eqref{eq: better estimate} in which $\bA_{> \mu}$ does not appear (i.e.~$b_i = 0$ for all~$i$), since $\bA_\mu$ and $\bA_{< \mu}$ both do not increase $(\mu + 1)$-rank. For monomials in which $\bA_{< \mu}$ does not appear ($a_i = 0$ for all~$i$), one can replace $\mu+1$ in~\eqref{eq: much better estimate} by $\mu-1$ based on a similar logic.

The situation becomes more complicated for monomials that feature all three operators $\bA_{< \mu}$, $\bA_\mu$ and $\bA_{> \mu}$, since then in general \emph{all} other $\nu$-ranks, $\nu \neq \mu$, will be increased as well. Using a remarkable construction, Arad et al.~\cite{Aradetal2013} showed that it is nevertheless beneficial to analyze the rank increase of the monomials with $k \ge \sqrt{\ell}$ in~\eqref{eq: expansion into monomials} starting from the above observation in order to arrive at an overall subexponential (with respect to $\ell$) estimate for the rank increase of all $\mu$-ranks under application of $\bA^\ell$, at least in the regime when $d$ is much larger than $\ell$. We state their main technical result as follows.

\begin{theorem}[Arad--Kitaev--Landau--Vazirani~\cite{Aradetal2013}]\label{th: result by Arad et al}
Consider a nearest neighbor interaction operator $\bA$ as in~\eqref{eq: nearest neighbor operator} and $\bu \in \bV$. Assume $d \ge 3$ and $2 \le n_\mu \le n$ for all $\mu = 1,\dots,d-1$. Assume further that it holds $\RI(\bA_\mu) \le R$ for all $\mu$ where $R \ge 1$. Then for any integers  $\ell \ge 1$ and $s \ge 2$ it holds that
\begin{equation}\label{eq: detailed rank estimate}
\max_\mu \rank_\mu(\bA^\ell \bu) \le \frac{\mathrm{e}^{s} n^s}{2(n-1)}  \left( 1 + \frac{\ell}{s} \right)^{s}  (1 + s)^{2 \ell/s + 1}    \left(\frac{\mathrm{e}^2 R}{4}\right)^{\ell/s} \max_\mu \rank_\mu(\bu),
\end{equation}
where the maxima are taken over $\mu = 1,\dots,d-1$.
\end{theorem}

The proof of Theorem~\ref{th: result by Arad et al} as presented in~\cite[Lemma~4.2]{Aradetal2013} contains several nontrivial ideas but is rather brief and does not feature the explicit prefactors stated above. Therefore, we provide a detailed version of the proof in Appendix~\ref{appendix: proof of theorem}. Note that following this proof, some of the prefactors in~\eqref{eq: detailed rank estimate} could be slightly improved, but our goal was to arrive at a sufficiently simple form.

The interesting case in the above theorem is when $d$ is very large compared to $s$ and $\ell$. Otherwise, the estimate may become trivial when the product of all prefactors exceeds $n^{d/2}$ (due to~\eqref{eq: trivial estimate}). With the case $d \to \infty$ in mind, Arad et al.~propose the choice $s \sim \sqrt{\ell}$ to obtain the desired result that $\mu$-ranks grow only sub-exponentially with respect to polynomial degree $\ell$.

\begin{corollary}\label{cor: corollary on rank increase}
Under the assumptions of Theorem~\ref{th: result by Arad et al} it holds that
\begin{equation}\label{eq: simplified rank bound}
\max_\mu \rank_\mu(\bA^\ell \bu) \le  \hat C_{\ell,n,R} \cdot  
\mathrm{e}^{ 3 \sqrt{\ell} \cdot  \ln \sqrt{\ell}} \cdot \max_\mu \rank_\mu (\bu)
\end{equation}
with
\[
\hat C_{\ell,n,R} = \frac{\mathrm e^6 n}{2(n-1)} (1+\sqrt{\ell})(2+ \sqrt{\ell}) \left( \frac{\mathrm{e}^3 n  R}{4} \right)^{\sqrt{\ell}}.
\]
\end{corollary}

\begin{proof}
For $\ell = 1$,~\eqref{eq: simplified rank bound} is trivially correct due to~\eqref{eq: naive rank estimate}. For $\ell > 1$, we can pick $\smash{s = \ceil{\sqrt{\ell}} \ge 2}$ in Theorem~\ref{th: result by Arad et al}. Then it holds $\ell / s \le \sqrt{\ell}$ and hence~\eqref{eq: detailed rank estimate} yields
\[
\max_\mu \rank_\mu(\bA^\ell \bu) \le \frac{(\mathrm{e} n)^{\ceil{\sqrt{\ell}}}}{2(n-1)}
\left(1 + \sqrt{\ell} \right)^{\ceil{\sqrt{\ell}}}
\left(1 + \ceil{\sqrt{\ell}}\right)^{2 \sqrt{\ell} + 1} \left( \frac{\mathrm{e}^2   R}{4} \right)^{\sqrt{\ell}} \max_\mu \rank_\mu(\bu).
\]
Due to $\ceil{\sqrt{\ell}} \le \sqrt{\ell} + 1$ this implies
\[
\max_\mu \rank_\mu(\bA^\ell \bu) \le \frac{\mathrm e n}{2(n-1)} \cdot (1+\sqrt{\ell}) \mathrm{e}^{\sqrt{\ell} \cdot \ln (1 + \sqrt{\ell})} (2 + \sqrt{\ell} ) \mathrm{e}^{2 \sqrt{\ell} \cdot \ln(2 + \sqrt{\ell})} \left( \frac{\mathrm{e}^3 n  R}{4} \right)^{\sqrt{\ell}} \max_\mu \rank_\mu(\bu).
\]
Applying the inequality $\ln(a + \sqrt{\ell}) \le \ln \sqrt{\ell} + a/\sqrt{\ell}$ (concavity of the logarithm) for $a = 1,2$ in the exponents, the asserted estimate follows.
\end{proof}

Regarding the actual rank-increase when applying a full polynomial $p_\ell(\bA) = \sum_{i=0}^\ell c_i \bA^i$ we do not attempt taking the different powers $\bA^i$ precisely into account, but confine ourselves with using the estimate~\eqref{eq: simplified rank bound} for the highest power simply for every term. Only for $i=0,1$ we may note that, by~\eqref{eq: local rank increase},
\[
\rank_\mu(c_1 \bA \bu + c_0 \bu ) \le (R + 3) \cdot \rank_\mu (\bu)
\]
for every $\mu$, which is better than the right side of~\eqref{eq: simplified rank bound} for any $\ell \ge 1$. Therefore, it is enough to apply the bound~\eqref{eq: simplified rank bound} only  $\ell$ times. As a result, we get that for $\ell \ge 1$ there holds
\begin{equation}\label{eq: rank increase in polynomial}
\max_\mu \rank_\mu(p_\ell(\bA) \bu) \le \ell \cdot \hat C_{\ell,n,R} \cdot  
\mathrm{e}^{ 3 \sqrt{\ell} \cdot  \ln \sqrt{\ell}} \cdot \max_\mu \rank_\mu(\bu),
\end{equation}
with $\hat C_{\ell,n,R}$ as in Corollary~\ref{cor: corollary on rank increase}.

Note that by~\eqref{eq: upper bound Kronecker rank}, the dependence on $R$ may be eliminated by replacing it with $n^2$, which, however, can be a quite rough estimate. It is also important to note that the constant $\hat C_{\ell,n,R}$ grows faster than $(nR)^{\sqrt{\ell}}$ for $\ell \to \infty$. However, in the asymptotic statements presented later it will be more convenient to hide this behavior in a larger exponent of the leading term in~\eqref{eq: rank increase in polynomial}. We state this as another corollary.

\begin{corollary}\label{cor: final generic estimate}
Let $\delta > 0$. Then under the assumptions of Theorem~\ref{th: result by Arad et al} it holds for all $\ell \ge 1$ that
\begin{equation*}\label{eq: more generic bound}
\max_\mu \rank_\mu(p_\ell(\bA) \bu) \le C_{\ell,n,\delta} \cdot \mathrm{e}^{(3+\delta)\sqrt{\ell} \cdot \ln \sqrt{\ell }} \cdot \max_\mu \rank_\mu(\bu) \le C_{n,\delta} \cdot \mathrm{e}^{(3+\delta)\sqrt{\ell} \cdot \ln \sqrt{\ell}} \cdot \max_\mu \rank_\mu(\bu)
\end{equation*}
with constants $C_{\ell,n,\delta} > 0$ that satisfy $\lim_{\ell \to \infty} C_{\ell,n,\delta} \to 0$ for any fixed $n$ and $\delta$. Consequently, one can take $C_{n,\delta} = \sup_{\ell} C_{\ell,n,\delta}$.
\end{corollary}
 
While with this ``trick'' we obtained a prefactor independent of $\ell$, the reader should bear in mind the asymptotic and purely qualitative nature of the second estimate. Note that from~\eqref{eq: rank increase in polynomial} we still expect $\lim_{n \to \infty} C_{n,\delta} = \infty$ for any fixed $\delta$.

\section{Implications for equations of interest}\label{sec: implications}

In the main part of this work we discuss choices of polynomials $p_\ell$ in~\eqref{eq: polynomial approximation} for approximating the solutions of linear problems~\eqref{eq: linear system} or~\eqref{eq: linear ODE}, and combine them with the rank increase estimates from the previous section. Throughout, we consider a nearest neighbor interaction operator~\eqref{eq: nearest neighbor operator} as introduced above and under the assumption $\RI(\bA_\mu) \le R$ for all $\mu$. Additional assumptions will be stated separately.

In Theorem~\ref{th: result by Arad et al} and some of the discussions following it we put some effort for also obtaining non-asymptotic final estimates like~\eqref{eq: rank increase in polynomial}, which is sensible since we are in a finite-dimensional setting. However, in order to get a clearer picture on low-rank approximability in the regime where $d$ is very large, we will now continue for convenience with a generic estimate of the form
\begin{equation}\label{eq: generic estimate}
\max_\mu \rank_\mu(p_\ell(\bA) \bu) \le  C_\ell \cdot \mathrm{e}^{D \sqrt{\ell} \cdot \ln \sqrt{\ell}} \cdot \max_\mu \rank_\mu(\bu), \qquad \text{$C_\ell \to 0$ for $\ell \to \infty$.} 
\end{equation}
By Corollary~\ref{cor: final generic estimate}, any choice $D > 3$ is feasible here. The sequence $(C_\ell)$ depends on the choice of~$D$, as well as on $n$ and $R$, which are now assumed to be fixed. More precise results than the ones below could be obtained by explicitly using~\eqref{eq: rank increase in polynomial} instead of~\eqref{eq: generic estimate} in the following considerations.

\subsection{Linear equations with low-rank right hand side}\label{sec: linear equations}

We consider the linear equation $\bA \bu = \bb$ where $\bA$ is an invertible linear operator on $\bV$ and $\bb \in \bV$. To obtain a polynomial approximation
\[
\bu_\ell = p_\ell(\bA) \bb
\]
to the solution $\bu_* = \bA^{-1} \bb$, we need to approximate the inverse $\bA^{-1}$ by a polynomial $p_\ell(\bA)$. By the Cayley--Hamilton an exact representation of the inverse by a polynomial of degree $\ell = \dim(\bV) - 1 \lesssim n^d$ is always possible, however, here we are of course interested in error estimates when using polynomials of smaller degree.

\paragraph{Hermitian positive definite operator} We first discuss the simplest but important case that $\bA$ is Hermitian and positive definite, that is, all its eigenvalues are positive. (The case of Hermitian negative definite $\bA$ admits the same results.) In this case, using spectral calculus, a polynomial approximation of $\bA^{-1}$ is obtained by approximating the function $f(\lambda) = 1/\lambda$ on the spectrum of $\bA$, leading to an error
\begin{equation}\label{eq: polynomial approximation for lin sys}
\| \bu_* - \bu_\ell \| \le \| \bA^{-1} - p_\ell(\bA) \| \cdot \| \bb \| = \max_{\lambda \in \sigma(\bA)} \abs{\frac{1}{\lambda} - p_\ell(\lambda)} \cdot \| \bb \|.
\end{equation}
It is a classical result of approximation theory that for any interval $[a,b]$ with $a > 0$ and $\ell \in \mathbb N$ there exists a unique polynomial $p_\ell$ of degree at most $\ell$ that is closest to $1/\lambda$ in uniform norm on $[a,b]$, the minimal error being
\[
\min_{\deg(p_\ell) \le \ell} \max_{\lambda \in [a,b]} \abs{\frac{1}{\lambda} - p_\ell(\lambda)} = \frac{1}{2} \left( \frac{1}{\sqrt{a}} + \frac{1}{\sqrt{b}} \right)^2 \left( \frac{\sqrt{b/a} - 1}{\sqrt{b/a} + 1} \right)^{\ell+1};
\]
see, e.g.,~\cite[Sec.~37]{Achieser1956},~\cite[Sec.~4.3]{Meinardus1967}, or~\cite[Theorem~4.2]{Jokar2005}. Therefore, if $0< \lambda_{\min}(\bA) \le \lambda_{\max}(\bA)$ denote the smallest and largest eigenvalues of $\bA$ and we use the optimal polynomials $p_\ell$ for approximating $1/\lambda$ on the whole interval $[\lambda_{\min}(\bA),\lambda_{\max}(\bA)]$, we obtain from~\eqref{eq: polynomial approximation for lin sys} a linear convergence rate
\begin{equation}\label{eq: rate for linear systems}
\| \bu_* - \bu_\ell \| \le \frac{1}{2} \left( \frac{1}{\sqrt{\lambda_{\min}(\bA)}} + \frac{1}{\sqrt{\lambda_{\max}(\bA)}} \right)^2 \cdot q_\bA^{\ell+1} \| \bb \| \le \frac{1}{2} \left( \sqrt{\kappa(\bA)} + 1 \right)^2  \cdot q_\bA^{\ell+1} \| \bu_* \|,
\end{equation}
with
\[
q_\bA = \frac{ \sqrt{\kappa(\bA)} - 1}{\sqrt{\kappa(\bA)} + 1} < 1.
\]
Here
\[
\kappa(\bA) = \frac{\lambda_{\max}(\bA)}{\lambda_{\min}(\bA)}
\]
is the spectral condition number of $\bA$. Hence, if only the endpoints of the spectrum of $\bA$ are known,~\eqref{eq: rate for linear systems} gives the optimal rate for polynomial approximation of $\bu_* = \bA^{-1} \bb$ when using a universal polynomial $p_\ell$ independent from the right-hand side $\bb$.

It is, however, important to note that the universal rate can be achieved even with a better constant by Krylov subspace methods taking $\bb$ into account. In fact, when initialized with residual $\bb$, the conjugate gradient (CG) method for solving the linear equation~\eqref{eq: linear system} produces after $\ell+1$ steps (implicitly) a polynomial approximation of the form $\bu_\ell^{\mathrm{cg}} = p_\ell^{\mathrm{cg}}(\bA) \bb$ that satisfies the well known estimate
\begin{equation}\label{eq: convenient CG}
\| \bu_* - \bu_\ell^{\mathrm{cg}} \|_\bA \le 2 \cdot q_\bA^{\ell+1} \|  \bu_* \|_\bA
\end{equation}
in $\bA$-norm $\| \bu \|_\bA \coloneqq \| \bA^{1/2} \bu \|$. Since the $\bA$-norm is often a more appropriate error measure for Hermitian positive definite linear equations anyway, we also use it to formulate our main result on such equations with nearest neighbor interaction operators and low-rank right hand sides, obtained from combining estimates~\eqref{eq: convenient CG} with~\eqref{eq: generic estimate}. An analogous result for the (given) standard norm in $\bV$ can be obtained by turning~\eqref{eq: convenient CG} into the $\bV$-norm estimate $\| \bu_* - \bu_\ell^{\mathrm{cg}} \| \le 2 \sqrt{\kappa(\bA)} \cdot q_\bA^{\ell+1} \|  \bu_* \|$ (which is better than~\eqref{eq: rate for linear systems}).

\begin{theorem}\label{th: main theorem for linear systems}
Let $\bu_* \in \bV$  be the solution of a linear equation $\bA \bu = \bb$ with $\bA$ being a Hermitian positive definite nearest neighbor interaction operator~\eqref{eq: nearest neighbor operator}. Then for any $0 < \varepsilon < 1$ there exists $\bu_\varepsilon = p_{\ell_\varepsilon}(\bA) \bb  \in \bV$, where $p_{\ell_\varepsilon}$ is a polynomial of degree
\[
 \ell_\varepsilon = \ceil{\frac{\ln \varepsilon}{\ln q_\bA}} - 1,
\]
satisfying
\begin{equation}\label{eq: accuracy linear systems}
\| \bu_* - \bu_\varepsilon \|_\bA \le 2 \varepsilon \| \bu_* \|_\bA,
\end{equation}
and
\begin{equation}\label{eq: rank bound linear systems}
\max_\mu \rank_\mu(\bu_\varepsilon) \le C_{\ell_\varepsilon}  \cdot \mathrm{e}^{D \sqrt{ \ell_{\varepsilon} } \cdot \ln \sqrt{ \ell_{\varepsilon} }} \cdot \max_\mu \rank_\mu(\bb)
\end{equation}
with the constants $C_\ell$ and $D$ from~\eqref{eq: generic estimate}, and $q_\bA$ from~\eqref{eq: rate for linear systems}. This estimate is specifically satisfied when~$\bu_\varepsilon$ is the result of $\ell_\varepsilon + 1$ steps of the CG method with initial residual $\bb$.
\end{theorem}

\begin{proof}
The choice of $\ell_{\varepsilon}$ in the CG method delivers $\bu_\varepsilon = p_{\ell_\varepsilon}^{\mathrm{cg}}(\bA) \bb$ with $q_{\bA}^{\ell_{\varepsilon} + 1} \le \varepsilon$ in~\eqref{eq: convenient CG}, and hence $\| \bu_* - \bu_\varepsilon \|_\bA \le 2 \varepsilon \| \bu_* \|_\bA$. The rank estimate~\eqref{eq: rank bound linear systems} follows directly from~\eqref{eq: generic estimate}.
\end{proof}

Since the convergence rate $q_\bA$ deteriorates with growing condition number of the operator~$\bA$, it is important to emphasize that the condition number of a sum $\bA = \sum_{\mu=1}^{d-1} \bA_\mu$, as considered in this work, does not grow with the system size $d$ as long as the summands $\bA_\mu$ are themselves Hermitian positive definite operators and their spectrum is located in a common fixed interval. For convenience we state this explicitly for the nearest neighbor interaction operators~\eqref{eq: nearest neighbor operator}.

\begin{proposition}\label{prop: condition of nearest neighbor}
Let the local operators $\tilde{A}_{\mu}$ in~\eqref{eq: local operator} be Hermitian and have eigenvalues in an interval $[\tilde \lambda_{\min}, \tilde \lambda_{\max}]$, with $0 < \tilde \lambda_{\min} \le \tilde \lambda_{\max}$ being independent of $\mu$. Then the nearest neighbor interaction operator~\eqref{eq: nearest neighbor operator} is Hermitian positive definite and satisfies $\kappa(\bA) \le \tilde \kappa \coloneqq \frac{\tilde \lambda_{\max}}{\tilde \lambda_{\min}}$. In particular, in this case $q_\bA$ in Theorem~\ref{th: main theorem for linear systems} satisfies $q_\bA \le \frac{\sqrt{\tilde \kappa} - 1}{\sqrt{\tilde \kappa} + 1} < 1$ independent of $d$.
\end{proposition}

\begin{proof}
Obviously, the $\bA_\mu$ are Hermitian and have the same spectrum as $\tilde{A}_\mu$. The statements then follow immediately from $(d-1)\tilde \lambda_{\min} \| \bu \|^2 \le \sum_{\mu = 1}^{d-1} \langle \bu, \bA_\mu \bu \rangle \le (d-1)\tilde \lambda_{\max} \| \bu \|^2$.
\end{proof}

The rank estimate~\eqref{eq: rank bound linear systems} allows us to make some conclusions on the behavior of the (relative) approximation errors
\begin{equation}\label{eq: tau A norm}
\tau_{\bA}^{(r)}(\bu_*) \coloneqq \min_{\max_\mu \rank_\mu(\bu) \le r} \frac{\| \bu_* - \bu \|_\bA}{\| \bu_* \|_\bA}
\end{equation}
in $\bA$-norm with respect to growing maximum $\mu$-rank $r$. For this we need to estimate the accuracy $\varepsilon$ in terms of the rank. Denoting $r_{\varepsilon} = \max_\mu \rank_\mu(\bu_\varepsilon)$ and setting $y_{\varepsilon} = \ln \sqrt{ \ell_{\varepsilon} }$,~\eqref{eq: rank bound linear systems} states that
\begin{equation}\label{eq:ln rank eps}
\ln r_{\varepsilon} \le \ln ( C_{\ell_\varepsilon} \cdot \max_\mu \rank_\mu(\bb)) + D \cdot y_{\varepsilon} \mathrm{e}^{y_\varepsilon}.
\end{equation}
The function $y \mapsto y \mathrm{e}^y$ is strictly monotone for $y \ge -\mathrm{e}^{-1}$ and possesses an inverse function $y = W(x)$ called the (principal branch of the) Lambert $W$-function. For $ x> \mathrm{e}$ it satisfies $W(x) \ge \ln\left( \frac{x}{ \ln x} \right)$ which is also the sharp asymptotic behaviour for growing $x$. Although more accurate estimates are available for~$W$, see, e.g.~\cite{Iacono2017}, we confine ourselves with this simpler one and obtain
\begin{equation}\label{eq: lower bound l}
\left(\frac{y_{\varepsilon} \mathrm{e}^{y_\varepsilon}}{ \ln (y_{\varepsilon} \mathrm{e}^{y_\varepsilon})} \right)^2 \le \mathrm{e}^{2 W(y_{\varepsilon} \mathrm{e}^{y_\varepsilon})} = \mathrm{e}^{2 y_{\varepsilon}} = \ell_{\varepsilon} \le \frac{\ln \varepsilon}{\ln q_{\bA}}
\end{equation}
for $\varepsilon$ small enough. Combining this estimate with~\eqref{eq:ln rank eps} is complicated by the fact that the precise behavior of $C_{\ell_{\varepsilon}} \to 0$ is not specified. However, asymptotically, since $y_{\varepsilon} \mathrm{e}^{y_\varepsilon} \to \infty$ for $\varepsilon \to 0$, we certainly can deduce
\begin{equation}\label{eq: asymptotic for LS}
\lim_{\varepsilon \to 0} \frac{\ln^{2 - \eta}(r_{\varepsilon})}{\abs{\ln \varepsilon}} = 0
\end{equation}
for any fixed $\eta > 0$. Since by~\eqref{eq: accuracy linear systems} we have $\tau_\bA^{(r_{\varepsilon})}(\bu_*) \le 2 \varepsilon$ for  $0 < \varepsilon < 1$,~\eqref{eq: asymptotic for LS} allows us to conclude
\begin{equation}\label{eq: asymptotic for LS ranks}
\liminf_{r \to \infty} \frac{\ln^{2 - \eta}(r)}{\abs{\ln \tau_\bA^{(r)}(\bu_*)}} = 0
\end{equation}
for a solution $\bu_*$ of a linear nearest neighbor interaction systems~\eqref{eq: linear system} under the made assumptions. It indicates a super-algebraic decay
\begin{equation}\label{eq: super algebraic decay LS}
\tau_{\bA}^{(r)}(\bu^*) \lesssim \mathrm{e}^{-c \ln^{2 - \eta}(r)} = r^{- c \ln^{1-\eta}(r)}
\end{equation}
for arbitrary fixed $c,\eta > 0$ and $r \to \infty$, although~\eqref{eq: asymptotic for LS ranks} only states this for a subsequence of (maximum) ranks $r$. The hidden constant in~\eqref{eq: super algebraic decay LS} may depend on the choice of $\eta$, $c$, and in addition on several other parameters, notably $\kappa(\bA)$ and $\max_\mu \rank_\mu(\bb)$. The latter should be ``small'' (compared to, say, $n^{d/2}$), otherwise the asymptotic $\lesssim$ statement becomes vacuous in finite-dimensional space. An error rate like~\eqref{eq: super algebraic decay LS} is also suggested by our numerical experiments; see Figure~\ref{fig: SLEs} in Section~\ref{sec: random linear systems}.

\paragraph{Non-Hermitian operator} Let us now consider the case of possibly non-Hermitian linear equations. Simply applying the previous result to the normal equation $\bA^* \bA \bu = \bA^* \bb$ is not an option since $\bA^* \bA$ will not be a nearest neighbor operator in general. Instead, polynomial approximation $p_\ell(\bA) \approx \bA^{-1}$ could again be invoked, but estimating the approximation error is much more difficult when $\bA$ is not Hermitian. We mention~\cite{Carson2024, Embree2025} as two recent surveys pointing to further references.

If $\bA$ is diagonalizable with $\bA = \bS {\bm \Lambda} \bS^{-1}$, where ${\bm \Lambda}$ is a (complex) diagonal matrix of eigenvalues, then it is immediate that
\[
\| \bA^{-1} - p_\ell(\bA) \| \le \kappa(\bS) \max_{\lambda \in \sigma(\bA)} \abs{\frac{1}{\lambda} - p_\ell(\lambda)}.
\]
In the more general case that zero is not contained in the numerical range
\[
W(\bA) = \left\{ \frac{\langle \bu,\bA \bu \rangle}{\| \bu \|^2} \colon \bu \neq 0 \right\}
\]
of $\bA$, the Crouzeix--Palencia theorem~\cite{Crouzeix2017} states that
\[
\| \bA^{-1} - p_\ell(\bA) \| \le (1 +\sqrt{2}) \max_{\lambda \in W(\bA)} \abs{\frac{1}{\lambda} - p_\ell(\lambda)}.
\]
Estimates of the form $\| \bu_* - p_\ell(\bA) \bb \| \le \varepsilon_\ell \| \bb \|$ on polynomial approximation hence could be obtained by invoking results on uniform approximation of $1/z$ on sets containing $\sigma(\bA)$ or $W(\bA)$, respectively, which however is not trivial. Once available, these could then be combined with the rank growth estimate~\eqref{eq: generic estimate} to obtain results of the same type as Theorem~\ref{th: main theorem for linear systems}.

Alternatively, one can invoke known convergence properties of practical Krylov methods for general linear equations and we will present one specific result. The generalized minimal residual (GMRES) method with initial residual $\bb$ constructs after $\ell + 1$ steps a polynomial approximation of the form $\bu_{\ell}^{\mathrm{gmres}} = p_\ell^{\mathrm{gmres}}(\bA) \bb$. A classic result presented in the original work~\cite{Saad1986} on GMRES and building on previous work by Elman states that if the Hermitian part $\frac{1}{2}(\bA + \bA^*)$ is positive definite, then
\begin{equation}\label{eq: estimate gmres}
 \| \bb -  \bA \bu_\ell^{\mathrm{gmres}} \| \le \hat q_\bA^{\ell+1} \| \bb \|, \qquad \hat q_\bA = \left( 1 - \frac{\lambda_{\min}^2(\tfrac{1}{2}(\bA^* + \bA))}{\| \bA \|^2} \right)^{1/2} < 1. 
\end{equation}
For an improvement of this bound, including an explicit treatment of complex matrices, see also~\cite{Beckermann2005}.

From~\eqref{eq: estimate gmres} we can formulate an analogous result to Theorem~\ref{th: main theorem for linear systems}, but now for the norm of the residual (which equals the error in $\bA^* \bA$ norm) instead of the error in $\bA$-norm. The condition that $\frac{1}{2}(\bA + \bA^*)$ is positive definite is equivalent with the real part of $\langle \bu, \bA \bu \rangle$ being positive for all $\bu \neq 0$. Such linear operators are sometimes also simply called positive definite even in the non-Hermitian case. They are necessarily invertible. Our result of course applies in the same way for operators that are negative definite in this broader understanding.

\begin{theorem}\label{th: theorem for non-hermitian linear systems}
Let $\bu_* \in \bV$  be the solution of a linear equation $\bA \bu = \bb$ with $\bA$ being a nearest neighbor interaction operator~\eqref{eq: nearest neighbor operator} such that $\bA + \bA^*$ is positive definite. Then for any $0 < \varepsilon < 1$ there exists $\bu_\varepsilon = p_{\ell_\varepsilon}(\bA) \bb  \in \bV$, where $p_{\ell_\varepsilon}$ is a polynomial of degree
\[
\ell_\varepsilon = \ceil{\frac{\ln \varepsilon}{\ln \hat q_\bA}} - 1,
\]
satisfying
\[
\| \bb - \bA \bu_\varepsilon \| \le \varepsilon \| \bb \|,
\]
and
\begin{equation*}\label{eq: rank bound linear systems non-hermitian}
\max_\mu \rank_\mu(\bu_\varepsilon) \le C_{\ell_\varepsilon}  \cdot \mathrm{e}^{D \sqrt{ \ell_{\varepsilon} } \cdot \ln \sqrt{ \ell_{\varepsilon} }} \cdot \max_\mu \rank_\mu(\bb)
\end{equation*}
with the constants $C_\ell$ and $D$ from~\eqref{eq: generic estimate}, and $\hat q_\bA$ from~\eqref{eq: estimate gmres}. This estimate is specifically satisfied when $\bu_\varepsilon$ is the result of $\ell_\varepsilon + 1$ steps of the GMRES method with initial residual $\bb$.
\end{theorem}

In analogy to Proposition~\ref{prop: condition of nearest neighbor} we also note that the definiteness assumption on $\bA$ can be enforced by assumptions on the local operators $\tilde A_\mu$ in~\eqref{eq: local operator}, which in addition yields a $d$-independent rate.

\begin{proposition}\label{prop: condition non-hermitian nearest neighbor}
Let the local operators $\tilde{A}_{\mu}$ in~\eqref{eq: local operator} satisfy $\lambda_{\min}( \tfrac{1}{2}(\tilde A_\mu + \tilde A_\mu^*)) \ge \bar \lambda_{\min}$ and $\| \tilde A_\mu \| \le \tilde \sigma_{\max}$ for some $0 < \bar \lambda_{\min} \le \tilde \sigma_{\max}$ independent of $\mu$. Then the nearest neighbor interaction operator~\eqref{eq: nearest neighbor operator} satisfies the assumptions of Theorem~\ref{th: theorem for non-hermitian linear systems} and it holds $\hat q_\bA \le ( 1 - \bar \lambda_{\min}^2 / \tilde \sigma_{\max}^2 )^{1/2} < 1$ independent of $d$.
\end{proposition}

\begin{proof}
It follows from~\eqref{eq: nearest neighbor operator} and~\eqref{eq: local operator} that $\lambda_{\min}(\tfrac{1}{2}(\bA + \bA^*)) \ge (d-1) \bar \lambda_{\min}$ and $\| \bA \| \le (d-1) \tilde \sigma_{\max}$.
\end{proof}

From Theorem~\ref{th: theorem for non-hermitian linear systems} we can infer approximability results in terms of the rank along the same lines as for Hermitian systems, but this time for the quantities
\begin{equation}\label{eq: tau residual norm}
\tau_{\bA^* \bA}^{(r)} \coloneqq \min_{\max_\mu \rank_\mu(\bu) \le r} \frac{\| \bb  - \bA \bu \|}{\| \bb \|}.
\end{equation}
At least for a subsequence of (maximum) ranks, the expected decay for growing $r$ is analogous to~\eqref{eq: asymptotic for LS ranks}, that is,
\[
\tau_{\bA^*\bA}^{(r)} \lesssim \mathrm{e}^{-c \ln^{2 - \eta}(r)} = r^{- c \ln^{1-\eta}(r)}
\]
for any $c,\eta > 0$, but again with a hidden constant that depends on the choices of $c$ and $\eta$ among others.

While the rates look similar, the fact that the error is measured using the relative residual (that is, the relative $\bA^* \bA$-norm of the error $\bu_* - \bu_\ell$) gives these results a different flavor than the corresponding results for Hermitian systems. Notably, inferring approximability in the original $\bV$-norm would naively use the estimate
\[
\frac{\| \bu_* - \bu_\ell \|}{\| \bu_* \|} \le \frac{\| \bA \|}{\sqrt{\lambda_{\min}(\bA^* \bA)}} \cdot \frac{\| \bb  - \bA \bu_\ell \|}{\| \bb \|}.
\]
However, different from the Hermitian case, it seems difficult to bound $\lambda_{\min}(\bA^* \bA)$ from below using assumptions on the local operators $\tilde A_\mu$. It is therefore less clear under which technical conditions the norm constant is guaranteed to remain bounded and ideally independent of~$d$.

\subsection{Linear autonomous ordinary differential equations}\label{sec: odes}

We now consider a linear autonomous ODE
\begin{equation*}
\frac{d}{dt} \bu(t) = \bA \bu(t), \qquad 0 \le t \le T, \qquad \bu(0) = \bu_0
\end{equation*}
on the space $\bV$. The linear operator $\bA$ does not necessarily need to be Hermitian. A standard assumption, however, is that all eigenvalues of $\bA$ have negative real part to ensure both dissipativity and asymptotic stability of the dynamics. In any case, the solution to~\eqref{eq: linear ODE} is given pointwise as
\[
 \bu_*(t) = \exp(t \bA) \bu_0.
\]
In Krylov subspace methods, the matrix exponential appearing in this formula is ultimately approximated by polynomials. We will again rely on classic results in this area.

The simplest practical numerical integration method for~\eqref{eq: linear ODE} is the explicit Euler method. If we fix the time $t$ as the final time for convenience and take $\ell$ time steps of length $t/\ell$, it constructs a polynomial approximation of the form
\[
\hat{\bu}_\ell(t) = \left(\mathbf{I} + \frac{t}{\ell} \bA\right)^{\ell} \bu_0,
\]
which for $\ell \to \infty$ converges to $\bu_*(t)$. However, the convergence of the sequence $(1 + \lambda/\ell)^{\ell}$ to $\mathrm{e}^{\lambda}$ underlying this argument is rather slow. The Taylor series of the exponential function provides much faster approximations in theory as follows. Consider scaled Taylor polynomials
\[
 p_{\vartheta,\ell}(\lambda) = \mathrm{e}^{-\vartheta} \sum_{k=0}^\ell \frac{\lambda^k}{k!}
\]
for $\vartheta \in \C$, and let
\[
 \bu_\ell(t) = p_{\vartheta,\ell}(t \bA) \bu_0.
\]
Then it holds, since $t\bA$ and $\vartheta \bI$ commute, that
\begin{align}
\| \exp(t \bA) - p_{\vartheta,\ell}(t \bA + \vartheta \bI) \| &= \| \mathrm{e}^{- \vartheta} \exp(t\bA + \vartheta \bI) - p_{\vartheta,\ell}(t \bA + \vartheta \bI) \| \notag \\ &\le \abs{\mathrm{e}^{- \vartheta}} \cdot \sum_{k = \ell + 1}^\infty \frac{\| t \bA + \vartheta \bI \|^k}{k!} \notag \\ 
&\le \abs{\mathrm{e}^{- \vartheta}} \cdot \mathrm{e}^{\| t \bA + \vartheta \bI \|} \cdot \frac{\| t \bA + \vartheta \bI \|^{\ell+1}}{(\ell+1)!}.\label{eq: Taylor for exp}
\end{align}
The second inequality might not be very sharp, but the estimate shows that polynomial approximations of the matrix exponential can converge super-exponentially fast with respect to $\ell$  in operator norm.

While one can always take $\vartheta = 0$ in~\eqref{eq: Taylor for exp}, it has the disadvantage that the remaining prefactor $\mathrm{e}^{\| t \bA \|}$ grows exponentially in time $t$. Moreover, the norm $\| \bA \|$ of a nearest neighbor interaction operator will typically grow linearly with system size $d$, so such a prefactor introduces an exponential dependence on~$d$ as well, which is something one actually hopes to avoid by using low-rank tensors. Even when the spectrum of $\bA$ is contained in the left half-plane, finding the optimal value for $\vartheta$ to minimize the above estimate is a non-trivial task in the case of non-Hermitian $\bA$. If, however, $\bA$ is Hermitian and negative semidefinite, then (in absence of additional assumptions on the eigenvalue of smallest modulus) the choice $\vartheta  = t \| \bA \| / 2$ minimizes the norm $\| t \bA + \vartheta \bI \|$ (which then actually equals $\vartheta$), and, moreover, makes the prefactor $\abs{\mathrm{e}^{- \vartheta}} \cdot \mathrm{e}^{\| t \bA + \vartheta \bI \|}$ in~\eqref{eq: Taylor for exp} disappear. Hence in the Hermitian negative semi-definite case we conclude that there exists a polynomial approximation $\bu_\ell(t)$ satisfying
\begin{equation}\label{eq: Taylor for exp symmetric}
\| \bu_*(t) - \bu_\ell(t) \|  \le \frac{\| t \bA \|^{\ell+1}}{2^{\ell + 1}(\ell+1)!} \| \bu_0 \|.
\end{equation}

The use of the Taylor polynomial remains of somewhat theoretical interest only. In practice, more stable Krylov subspace methods are advisable for the approximation of the matrix exponential applied to~$\bu_0$. In the classic work~\cite{Saad1992} it is shown that the approximation
\begin{equation}\label{eq: Krylov subspace for exp}
\bu_\ell(t) = \| \bu_0 \| \mathbf{V}_{\ell + 1} \exp(t \mathbf{H}_{\ell+1}) \mathbf{e}_1,
\end{equation}
where $\mathbf{V}_{\ell + 1}$ is an orthonormal basis of the Krylov subspace~\eqref{eq: Krylov subspace} and $\mathbf{H}_{\ell+1}$ is the corresponding upper Hessenberg matrix obtained from Arnoldi's procedure with starting vector $\bu_0$ (and $\mathbf e_1$ is a unit vector), achieves, up to a factor two, the same rates~\eqref{eq: Taylor for exp} and~\eqref{eq: Taylor for exp symmetric} in the Hermitian case. The references~\cite{DruskinKnizhnerman1989,StewartLeyk1996,HochbruckLubich1997} contain improved estimates under different assumptions on $\bA$. In the following we will use the results of~\cite[Theorems~2 and~5]{HochbruckLubich1997}: assume the numerical range of $\bA$ is contained in a disc $\abs{z+\rho} \le \rho$, then for $\ell + 1 \ge 2\rho t$ the error of the Krylov subspace approximation~\eqref{eq: Krylov subspace for exp} satisfies
\begin{equation}\label{eq: Krylov subspace for exp general}
\| \bu_*(t) - \bu_\ell(t) \| \le E \mathrm{e}^{- \rho t} \left( \frac{\mathrm{e} \rho t}{\ell + 1} \right)^{\ell + 1} \| \bu_0 \|
\end{equation}
with a constant $E = E(t) = 10(\rho t)^{-1}$ when $\bA$ is Hermitian, and $E=12$ in the general case. Note that~\cite[Theorem~2]{HochbruckLubich1997} also provides a rate $\| \bu_*(t) - \bu_\ell(t) \| \le 10 \mathrm{e}^{- (\ell+1)^2 / (5\rho t)} \| \bu_0 \|$ in the Hermitian case in the range $2 \sqrt{\rho t} \le \ell+1\le 2 \rho t$, which however applies to the non-Hermitian case only under additional technical conditions on the numerical range~\cite[Theorem~6]{HochbruckLubich1997}. In order to avoid too many cases we just proceed with~\eqref{eq: Krylov subspace for exp general} and confine ourselves more or less with the ``asymptotic'' behavior for $\ell \to \infty$. More references on Krylov subspace approximation of the matrix exponential can be found in~\cite{MolerVanLoan2003}.

When applying~\eqref{eq: Krylov subspace for exp general} to nearest neighbor operators of the form~\eqref{eq: nearest neighbor operator} one should note that the numerical radius $\rho$ might depend linearly on $d$. In particular, if the numerical range of each local operator $\tilde{A}_\mu$ in~\eqref{eq: nearest neighbor operator} is contained in a disc $\abs{z + \tilde{\rho}} \le \tilde{\rho}$, then a priori the numerical range of $\bA$ is contained in the disc $\abs{z + \rho} \le \rho$ with $\rho = d \tilde \rho$. This introduces a potentially unfavorable dependence on the dimension $d$, which however in our experiments had very minor impact (cf.~Sections~\ref{sec: Ising model} and~\ref{sec: catalytic CO oxidation}).

By combining the error decay~\eqref{eq: Krylov subspace for exp general} with the generic rank estimate~\eqref{eq: generic estimate}, we can obtain the following approximability result for linear ODEs.

\begin{theorem}\label{thm: main result ODEs}
Let $\bu_*(t) = \exp(t \bA) \bu_0$ be the solution of the differential equation~\eqref{eq: linear ODE} at time $t$ with $\bA$ being a nearest neighbor interaction operator~\eqref{eq: nearest neighbor operator}. Assume the numerical range of $\bA$ is contained in a disc $\abs{z + \rho} \le \rho$. In particular, if the numerical range of each local operator $\tilde{A}_\mu$ is contained in a disc $\abs{z + \tilde{\rho}} \le \tilde \rho$, then $\rho = d \tilde \rho$ is feasible. Let $E = 10(\rho t)^{-1}$ if $\bA$ is Hermitian, or $E=12$ in the general case.

Then for any $0 < \varepsilon \le E \cdot (\mathrm e / 4)^{\rho t}$ there exists $\bu_\varepsilon(t) = p_{\ell_\varepsilon}(t\bA) \bu_0  \in \bV$, where $p_{\ell_\varepsilon}$ is a polynomial of degree $\ell_\varepsilon$ with
\[
2 \rho t - 1 \le \ell_\varepsilon \le F_{\varepsilon,\rho t} \cdot \left[ \frac{ \ln (E/\varepsilon)}{\ln \left(\frac{\ln (E/\varepsilon)}{\mathrm e \rho t} \right)} \right] - 1, \qquad F_{\varepsilon,\rho t} = \left[ \ln\left(\frac{\ln (E/\varepsilon)}{\mathrm e \rho t} \right) \right]^{\frac{\mathrm e}{(\mathrm e-1) \ln \left(\frac{\ln (E/\varepsilon)}{\mathrm e \rho t} \right)}},
\]
such that
\[
\frac{\| \bu_*(t) - \bu_\varepsilon(t) \|}{\| \bu_0 \|} \le \varepsilon
\]
and
\[
\max_\mu \rank_\mu(\bu_\varepsilon(t)) \le C_{\ell_\varepsilon}  \cdot \mathrm{e}^{D \sqrt{ \ell_{\varepsilon} } \cdot \ln \sqrt{\ell_{\varepsilon}}} \cdot \max_\mu \rank_\mu(\bu_0)
\]
with the constants $C_\ell$ and $D$ from~\eqref{eq: generic estimate}. 
\end{theorem}

Note that $F_{\varepsilon,\rho t} \to 1$ for $\varepsilon \to 0$ and $\rho t$ fixed.

\begin{proof}
In the range $\ell + 1 \ge 2 \rho t$, the right-hand side of~\eqref{eq: Krylov subspace for exp general} is monotonically decreasing in $\ell$ and achieves any values in the interval $(0, E \cdot (\mathrm e / 4)^{\rho t} \| \bu_0 \|]$. For estimating the number of iterations to reach a particular value $\varepsilon \| \bu_0 \|$ in that interval, one has to resolve the inequality
\[
a \left( \frac{b}{\ell + 1} \right)^{\ell + 1} \le \varepsilon, \qquad a = E \mathrm{e}^{-\rho t}, \quad b = \mathrm{e}\rho t,
\]
for the smallest possible $\ell + 1 \ge 2 \rho t$. By first rewriting it as
\[
y \mathrm{e}^{y} \ge b^{-1} \ln (a / \varepsilon), \qquad y = \ln\left( \frac{\ell+1}{b} \right),
\]
the inequality can be expressed as $y \ge W (b^{-1} \ln (a / \varepsilon))$ using the Lambert $W$-function (cf.~the discussion following Theorem~\ref{th: main theorem for linear systems}), which in turn yields
\begin{equation}\label{eq: polynom degree ODE}
\ell + 1 = b \mathrm{e}^{y} \ge b \mathrm{e}^{W (b^{-1} \ln (a / \varepsilon))}. 
\end{equation}
Note that the Lambert $W$-function requires $b^{-1} \ln (a / \varepsilon) \ge -\frac{1}{\mathrm e}$, which is indeed satisfied because
\begin{equation}\label{eq:1}
b^{-1} \ln (a / \varepsilon) = \frac{\ln (E/\varepsilon)}{\mathrm e \rho t} - \frac{1}{\mathrm e}
\end{equation}
and we considered $\varepsilon \le E \cdot (\mathrm e / 4)^{\rho t}$, which is strictly smaller than $E$ and hence $\ln(E / \varepsilon) > 0$. In light of~\eqref{eq:1}, since $W$ is monotonically increasing in the considered range, a sufficient condition for~\eqref{eq: polynom degree ODE} is
\[
\ell+1 \ge \mathrm{e}\rho t \cdot \mathrm{e}^{W(\frac{\ln (E/\varepsilon)}{\mathrm e \rho t})}
\]
which also automatically ensures our assumption $\ell + 1 \ge 2 \rho t$ (since $W(\frac{\ln (E/\varepsilon)}{\mathrm e \rho t}) \ge 0$) made at the beginning of the proof. Using the known estimate $W(x) \le \ln \left(\frac{x}{\ln x} \right) + \frac{\mathrm e}{\mathrm e - 1} \left( \frac{\ln \ln x}{\ln x} \right)$ (see, e.g.,~\cite[Sec.~4.3]{Iacono2017}), an even stronger sufficient condition is
\begin{equation*}\label{eq: sufficient condition}
\ell+1 \ge F_{\varepsilon,\rho t} \cdot \left[ \frac{ \ln (E/\varepsilon)}{\ln \left(\frac{\ln (E/\varepsilon)}{\mathrm e \rho t} \right)} \right]
\end{equation*}
with $F_{\varepsilon,\rho t}$ as stated in the theorem. Choosing $\ell_\varepsilon$ as the smallest integer satisfying this will hence guarantee that the right hand side of the error estimate~\eqref{eq: Krylov subspace for exp general} is less than $\varepsilon \| \bu_0 \|$. The asserted rank bound then follows from~\eqref{eq: generic estimate}.
\end{proof}

The above result allows us to derive an asymptotic rate for low-rank approximability of solutions of ODEs with nearest neighbor interaction operators. Notably, this rate which is strictly better than the asymptotic rates for linear equations derived from Theorems~\ref{th: main theorem for linear systems} and~\ref{th: theorem for non-hermitian linear systems}. For fixed time $t$, let $\ell_\varepsilon$ be the polynomial degree as in Theorem~\ref{thm: main result ODEs} and set $r_{\varepsilon} = \max_\mu \rank_\mu(\bu_\varepsilon(t))$ and $y_\varepsilon = \ln \sqrt{\ell_\varepsilon}$ similar as before. Instead of~\eqref{eq: lower bound l}, we now obtain
\[
\left(\frac{y_{\varepsilon} \mathrm{e}^{y_\varepsilon}}{ \ln (y_{\varepsilon} \mathrm{e}^{y_\varepsilon})} \right)^2 \le \ell_\varepsilon \lesssim \frac{\abs{\ln \varepsilon}}{\ln \abs{\ln \varepsilon}}
\]
(for $\varepsilon \to 0$) where the first inequality is derived in the same way as in~\eqref{eq: lower bound l}, and the second is due to the theorem. Combining this with the corresponding version of~\eqref{eq:ln rank eps} (with $\bu_0$ instead of $\bb$) shows
\[
\lim_{\varepsilon \to 0} \frac{\ln^{2 - \eta}(r_{\varepsilon})}{\abs{\ln \varepsilon}} \cdot \ln \abs{\ln \varepsilon} = 0
\]
for any $\eta > 0$. Using that $\ln \abs{\ln \varepsilon} / \abs{\ln \varepsilon}$ is monotonically decreasing for $\varepsilon \to 0$,  this implies
\begin{equation}\label{eq: lowrank rate for exp}
\liminf_{r \to \infty} \, \frac{\ln^{2 - \eta}(r)}{\abs{\ln \tau^{(r)}(\bu_*(t))}} \cdot \ln \abs{\ln \tau_r(\bu_*(t))}  = 0
\end{equation}
for the low-rank approximation numbers of $\bu_*(t)$ in $\bV$-norm as defined in~\eqref{eq: approximation numbers}. This indicates that the low-rank approximability of $\bu_*(t)$ is asymptotically better than $\mathrm{e}^{-c \ln^{2-\eta}(r)}$, which is also suggested by some of our numerical results; see in particular Figure~\ref{fig: Ising 2}. However, we remind again, that the actual non-asymptotic rate might deteriorate with growing $t$, and possibly also with $d$ (in the non-Hermitian case) according to the remarks preceding Theorem~\ref{thm: main result ODEs}. This aspect may require additional consideration.

\section{Numerical experiments}\label{sec: applications}

In this section we illustrate the theoretical approximability estimates for nearest neighbor interactions systems with three numerical experiments.
All experiments as well as the tensor-based algorithms employed in them have been implemented in Python and are collected in the toolbox \texttt{Scikit-TT}, available on GitHub at \url{https://github.com/PGelss/scikit_tt}.

\subsection{Random linear equation}\label{sec: random linear systems}

We begin with an experiment on TT-based approximation of solutions to systems of linear equations. By constructing random systems $\mathbf{A} \mathbf{u} = \mathbf{b}$ where $\bu \in \R^{n^d} \cong \R^n \otimes \dots \otimes \R^n$ ($d$~times), we want to validate the theoretical error bounds derived in Section~\ref{sec: linear equations}. In order to build a random nearest neighbor operator~$\mathbf{A}$, we use the SLIM decomposition, a form of TT decomposition for specific nearest neighbor interaction systems introduced in~\cite{GELSS2017}. That is, we consider operators with a canonical representation of the form
\begin{equation}\label{eq: SLIM decomposition}
\begin{split}
\bA =& S \otimes I \otimes \dots \otimes I + \dots + I \otimes \dots \otimes I \otimes S \\
     &+ L \otimes M \otimes I \otimes \dots \otimes I + \dots + I \otimes \dots \otimes I \otimes L \otimes M,
\end{split}
\end{equation}
where $S, L, I, M \in \mathbb{R}^{n \times n}$ with $I$ denoting the identity matrix. Note that for such systems, the Kronecker rank of the local operators $\bA_\mu$ is (at most) three, that is, $R=3$ in~\eqref{eq: generic RI A_mu estimate}, since $\tilde A_\mu = \tfrac{1}{2}S \otimes I + I \otimes \tfrac{1}{2}S + L \otimes M$ for $\mu = 2,\dots,d-2$ and similar (with one factor $\tfrac{1}{2}$ omitted) for $\mu=1,d-1$. We ensure that the resulting operator $\mathbf{A}$ is symmetric positive definite by choosing the matrices $S, L, M$ symmetric positive definite.

To tackle computations in the TT format, like solving systems of linear equations or eigenvalue problems, a range of algorithms are available. Among these, the Alternating Linear Scheme (ALS)~\cite{HOLTZ2012} stands out as a foundational approach. ALS operates by iteratively solving low-dimensional local systems of the same type (linear system or eigenvalue problem) for each TT core. Thanks to the SLIM decomposition of $\bA$, these local systems can be set up efficiently despite the overall large dimensionality, and are then solved using standard numerical methods.  ALS proceeds by updating the TT cores sequentially through bidirectional half sweeps, maintaining bounded TT ranks, typically with a common fixed rank bound $r$ for all cores, throughout the iteration, which is a defining characteristic of the algorithm.

For our experiments, we set $n=4$ and define the right-hand side as a rank-one tensor 
\begin{equation*}
\bb = b^{(1)} \otimes b^{(2)} \otimes \dots \otimes b^{(d)},
\end{equation*}
where each vector $b^{(\nu)} \in \mathbb{R}^n$ has random entries between $0$ and $1$. To ensure that $S$, $L$, and $M$ are symmetric positive definite matrices, we create random orthogonal matrices $U \in \mathbb{R}^{n \times n}$ and diagonal matrices $D$ with $D_{i,i} > 0$ for $i=1, \dots, n$, and define the SLIM components respectively as $U^\top \, D \, U$.

After running ALS with $10$ sweeps, we obtain an approximate solution $\mathbf{u}$ and compute the relative residual error
\begin{equation*}
e_{\text{sol}} = \frac{\lVert \mathbf{A} \mathbf{u} - \mathbf{b}\rVert_2}{\lVert \mathbf{b} \rVert_2}.
\end{equation*}
To analyze the error behavior with respect to ranks, we run the experiment for different maximum TT ranks $r$. For each choice, the experiment is repeated 100 times. The results for $d=10$ and $d=20$ are shown in Figure~\ref{fig: SLEs}. It is clearly seen that for all repetitions the error not only decreases with increasing $r$, but also is bounded by $f(r) = \mathrm{e}^{-\ln^2(r)} = r^{- \ln r}$, which is in line with the suggested asymptotic bound~\eqref{eq: super algebraic decay LS}. Additionally, the mean of the approximation errors behaves approximately like a scaled version of $f(r)$.

\begin{figure}[t]
    \centering
    \begin{subfigure}[b]{0.45\textwidth}
        \centering
        \includegraphics[height=140px]{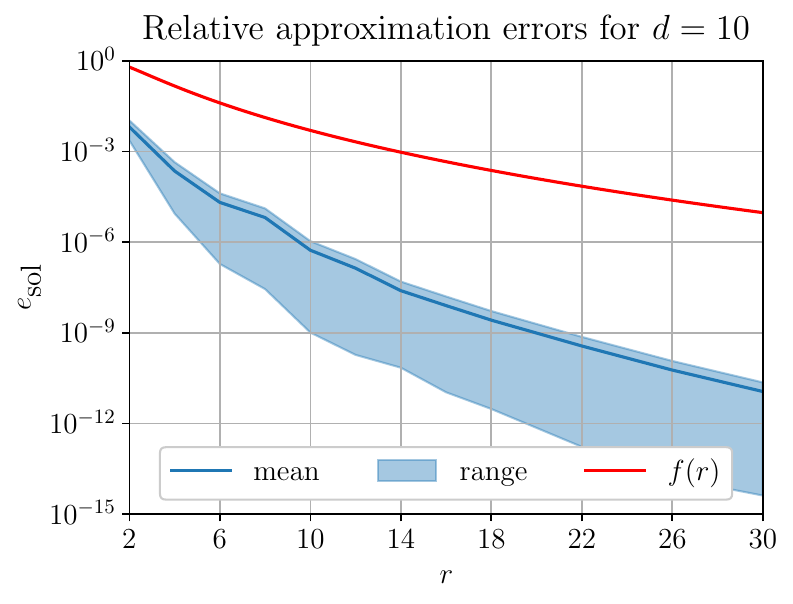}
        \caption*{~~(a)}
    \end{subfigure}
    \hfill
    \begin{subfigure}[b]{0.45\textwidth}
        \centering
        \includegraphics[height=140px]{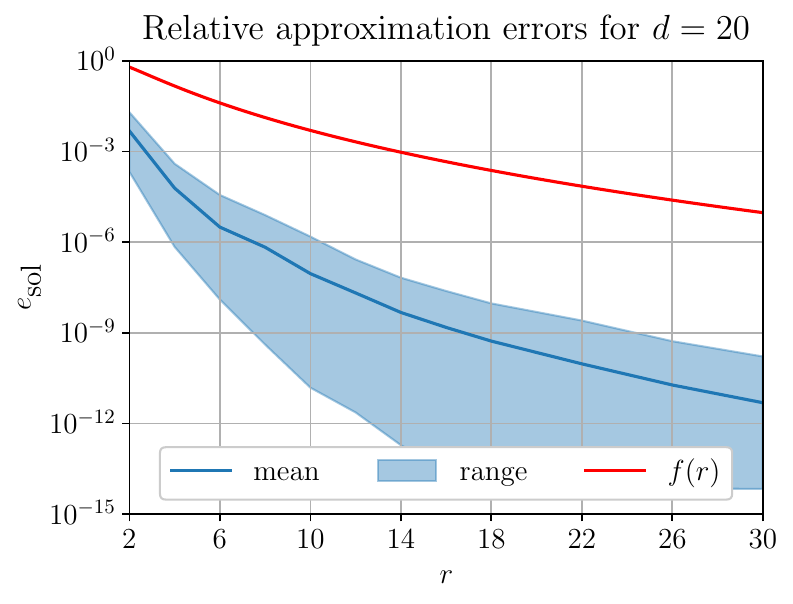}
        \caption*{~~(b)}
    \end{subfigure}
    \caption{Systems of linear equations: Relative approximation errors of the solutions of systems of linear equations computed by using the TT approach for dimensions (a) $d=10$ and (b) $d=20$ and TT ranks limited to different values $r$. The results are presented in the form of the mean error (blue line) and the range (blue area) where all errors lie. In order to illustrate consistency with the theoretical results, we also plot the function $f(r) = \mathrm{e}^{-\ln^2(r)}$. }
    \label{fig: SLEs}
\end{figure}

\subsection{Ising model}\label{sec: Ising model}

As a second example, we consider the well known one-dimensional stochastic Ising model from statistical mechanics \cite{GLAUBER1963,LENZ1920, ISING1925}. The system describes a Markov process on a chain of spins, each of which can be in one of the two states $+1$ and $-1$. The state configuration of the whole chain of length $d$ is hence fully characterized by a vector $\sigma = (\sigma_1 , \dots , \sigma_d)^\top \in \mathcal{S} := \{+1,-1\}^d$. In the Ising model, transitions are only possible between pairs of configurations that differ in the state of a single spin. Specifically, for all $\sigma \in \mathcal{S}$ and $t \in \mathbb{R}^+$, the probabilities $p(\sigma, t)$ for the system being in configuration $\sigma$ at time point $t$ obey the Markovian master equation~\cite{VanKampen2007}
\begin{equation}\label{eq: master equation}
\frac{\partial}{\partial t} p(\sigma,t) = \sum_{s \in \mathcal{S}} p(s,t) A({s,\sigma}), 
\end{equation}
with the infinitesimal generator $A$ given by
\begin{equation}\label{eq: Ising - transition rates} 
A(s, \sigma) =
\begin{cases}
0, &\text{if}~\left\|\sigma - s \right\|_1 > 1,\\
\mathrm e^{- \frac{\beta}{2} (H(s) - H(\sigma))}, &\text{if}~\left\|\sigma - s \right\|_1 = 1,\\
- \sum_{s' \neq s} A(s, s'), &\text{if}~\left\|\sigma - s \right\|_1 = 0,
\end{cases}
\end{equation}
where
\begin{equation}\label{eq: Ising - Hamiltonian function}
H(\sigma) = - \sum_{\mu=1}^{d-1}  \sigma_\mu \sigma_{\mu+1} -  \sum_{\mu=1}^d  \sigma_\mu
\end{equation}
is the Hamiltonian,  $\beta = 1/(k_B \cdot T)$ is the inverse temperature with Boltzmann constant $k_B$ and thermodynamic temperature $T$, and $\| \cdot \|_1$ denotes the $1$-norm. Note that we omitted the dependence of $p(\sigma, t)$ on an initial state for the sake of simplicity. The defined process satisfies the detailed balance condition and is known to relax to the unique stationary distribution
\begin{equation}\label{eq: Ising - configuration probability}
\pi(\sigma) = \frac{\mathrm e^{-\beta H(\sigma)}}{Z},
\end{equation}
where $Z = \sum_{\sigma \in \mathcal{S}} \mathrm e^{-\beta H(\sigma)}$ is a normalization constant.

By regarding the configurations $\sigma \in \mathcal S$ as multi-indices, we can arrange the $p(\sigma,t)$ as the entries in an order $d$-array $\mathbf p(t) \in \mathbb R^{2 \times \dots \times 2}$ and~\eqref{eq: master equation} becomes a linear ODE $\frac{d}{dt} \mathbf p(t) = \bA^\top \mathbf p(t)$. This system is not quite a nearest neighbor system, but can be turned into one by gluing neighboring states together and treat the whole configuration as a tensor product of order $d/2$ in $(\mathbb R^2 \otimes \mathbb R^2) \otimes \dots \otimes (\mathbb R^2 \otimes \mathbb R^2)$. With respect to this decomposition, the operator $\bA$ admits a SLIM decomposition~\eqref{eq: SLIM decomposition} and in particular is of nearest neighbor type. We refer to Appendix~\ref{app: construction of W}, where the canonical as well as a TT representation of the generator are given. We hence expect low-rank approximability of the configurations $\mathbf p(t)$ and in fact, it can be shown that the stationary distribution $\pi$ can be expressed as a rank-$2$ TT decomposition, which can be found in Appendix~\ref{app: construction of pi}.

In the theoretical part we used polynomial approximation of the matrix exponential to derive approximability results in terms of $\mu$-ranks. However, in our numerical simulation of the Ising model we employ the \emph{implicit} Euler method for the time discretization of the master equation~\eqref{eq: master equation}. It leads to a sequence of (symmetrized) linear problems of the form
\begin{equation}\label{Ising - symmetric SLE}
\left( \mathbf{I} - \tau \mathbf{A}^\top\right)^\top \left( \mathbf{I} - \tau \mathbf{A}^\top \right) \mathbf{p}_k = \left( \mathbf{I} - \tau \mathbf{A}^\top \right)^\top \mathbf{p}_{k-1}, \quad k>0,
\end{equation}
where $\mathbf{p}_k$ are the probabilities at the $k$-th time step. We seek for approximations of $\mathbf p_k$ in low-rank TT tensor format. Since the discretized generator $\mathbf{A}$ admits a nice TT representation (see Appendix~\ref{app: construction of W}), the linear equations~\eqref{Ising - symmetric SLE} can be solved approximately in the TT format by utilizing ALS with prescribed TT ranks~\cite{GELSS2016,GELSS2017b}. We will prescribe the same rank bound $r$ for all TT ranks. This approach does not guarantee the preservation of nonnegativity or the 1-norm of the approximate solutions. However, to avoid distorting the simulation results, we do not artificially normalize the intermediate steps here.

For the numerical experiments, we set the inverse temperature $\beta$ to $0.1$ and employ an initial condition where the configuration with all spins equal to $+1$ has a probability of one. The numerical integration is done using a constant time step of $\tau = 0.1$ and we repeat the ALS 10 times for each system of linear equations. Figure~\ref{fig: Ising 1} (a) shows obtained Euclidean norms of the time derivatives $\mathbf{A}^\top  \mathbf{p}_k$ for varying number of spins $d$, but with all TT ranks limited to $r =20$ in ALS. After a short induction time of approximately $25$ steps, all solutions relax exponentially to the stationary distribution. After $250$ steps, the derivatives are close to machine precision and do not change anymore.

\begin{figure}[t]
    \quad
    \begin{subfigure}[b]{0.45\textwidth}
        \centering
        \includegraphics[height=140px]{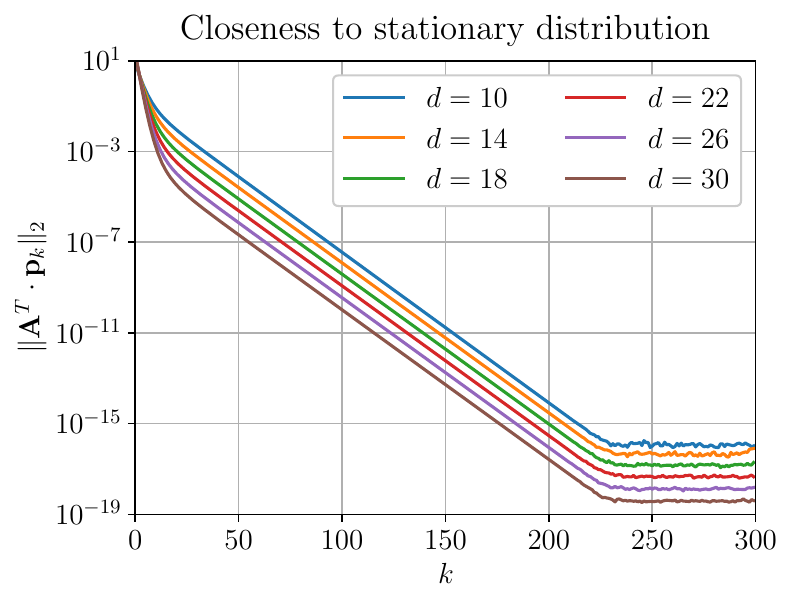}
        \caption*{~~(a)}
    \end{subfigure}
    \hfill
    \begin{subfigure}[b]{0.45\textwidth}
        \centering
        \includegraphics[height=140px]{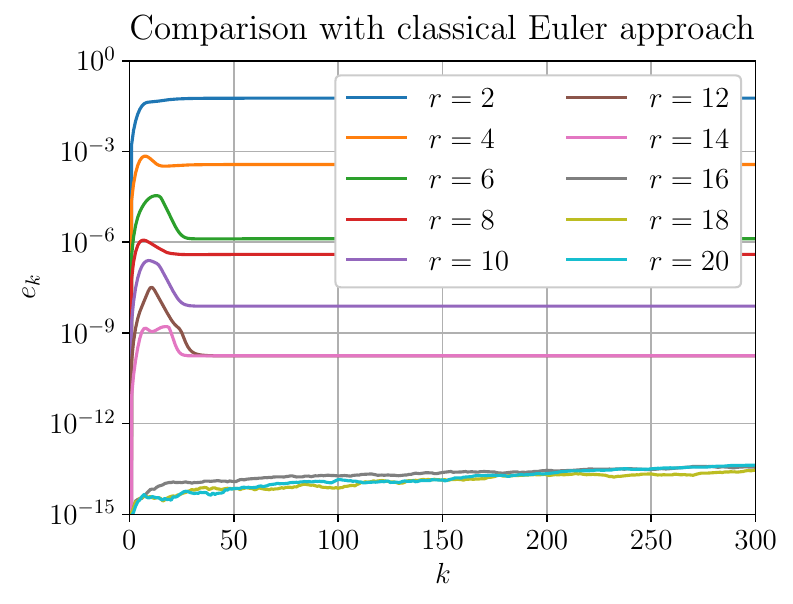}
        \caption*{~~(b)}
    \end{subfigure}
    \quad
    \caption{Numerical time integration of the Ising model: (a) Norm of the derivatives $\mathbf{A}^\top \mathbf{p}_k$ at the different time steps of the tensor-based implicit Euler method with TT rank bound 20. (b) Relative error between tensor-based Euler method and classical Euler method for different TT ranks and $d=10$.}
    \label{fig: Ising 1}
\end{figure}

To assess the accuracy of the low-rank approach, we also compare the results to a classical full Euler approach. For $d=10$ we turn the TT representation of $\bA$ into a $2^d \times 2^d$ matrix and apply the Euler method~\eqref{Ising - symmetric SLE} (with the same step size) without rank constraints, resulting in a trajectory represented by vectors $p_k \in \R^{2^{10}}$, $k =0 , \dots , 300$. Note that for larger values of $d$ this would quickly become infeasible. Afterwards, we compute the errors
\begin{equation*}\label{eq: error formula}
e_k = \frac{\left\| \textrm{vec}(\mathbf{p}_k) - p_k \right\|_2}{\left\| p_k \right\|_2}
\end{equation*}
for all steps. For a detailed description of vectorizations and matricizations of tensor trains, we refer to~\cite{GELSS2017b}. Figure~\ref{fig: Ising 1} (b) shows a comparison for different TT rank bounds $r$ of the low-rank solutions $\mathbf{p}_k$. As one can see, the difference between the distributions computed with the TT-based implicit Euler method and the classical method approaches machine precision. For TT ranks larger than 14, the error $e_k$ is essentially the same, irrespective of the rank. Thus, within the implicit Euler discretization, the TT approximation can be regarded as exact.

Finally, we consider the error of the computed ``stationary'' distribution after 300 time steps to the true stationary distribution $\pi$ given in~\eqref{eq: Ising - configuration probability} for different dimensions $d$ and TT rank bounds $r$. We compute the relative error
\begin{equation*}
e_\pi = \frac{\lVert \textrm{vec}(\mathbf{p}_{300}) - \pi \rVert_2}{\lVert \pi \rVert_2}
\end{equation*}
for dimensions up to $d=30$. Figure~\ref{fig: Ising 2} shows that, as expected, the approximation error $e_\pi$ decreases for growing TT ranks. In particular, the relative error between the approximated distribution and the exact stationary distribution is already below $0.1\%$ for ranks larger than 10. What is even more interesting is that we can validate the theoretical error bounds derived in Section~\ref{sec: odes}: the rate of convergence indicated by Figure~\ref{fig: Ising 2} is strictly faster than $\mathrm{e}^{-\ln^2(r)^2}$, which supports~\eqref{eq: lowrank rate for exp}. However, as suspected in the discussion following~\eqref{eq: lowrank rate for exp}, the rate seems to slightly deteriorate when $d$ increases.

\begin{figure}[t]
    \centering
    \includegraphics[height=140px]{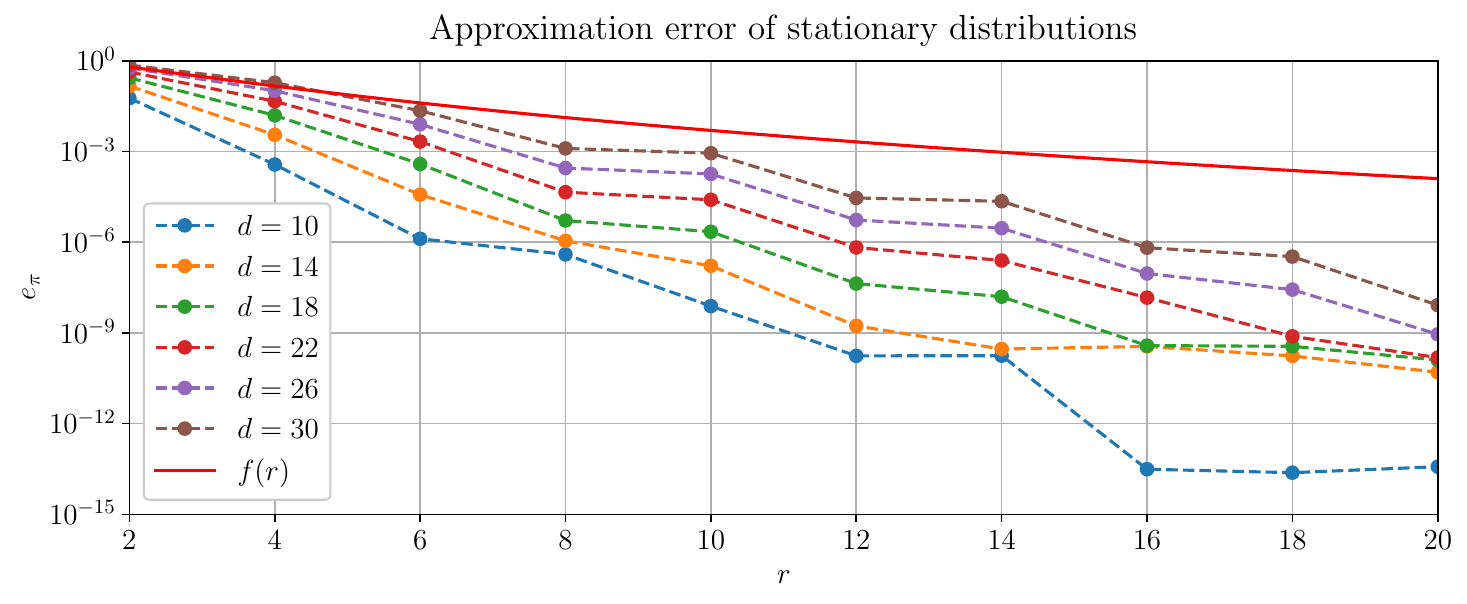}
    \caption{Approximation of stationary distributions: Relative error between stationary distributions computed with the TT approach and exact stationary distributions for different dimensions and TT rank bounds. To illustrate consistency with the theoretical results, we also plot the function $f(r) = \mathrm{e}^{-\ln^2(r)}$.}
    \label{fig: Ising 2}
\end{figure}

\subsection{Catalytic CO oxidation}\label{sec: catalytic CO oxidation}

As a last experiment, we investigate a simple kinetic model for catalytic CO oxidation, cf.~\cite{GELSS2016}. Again, the process is modeled as a Markov jump process of a chain of units each having only a finite set of states. In the context of heterogeneous catalysis the units correspond to the adsorption sites on the catalytic surface and the states correspond to which species is adsorbed on the respective site. Each site may be in three different states (1 = empty, 2 = O-covered, 3 = CO-covered). The possible events are the unimolecular adsorption/desorption of CO (i.e.~only a single site changes its state), the dissociative oxygen adsorption/desorption on two neighboring sites, the diffusion of adsorbed CO or O to a neighboring site, and the formation of gaseous CO$_2$ from adsorbed CO and O on neighboring sites. We employ periodic boundary conditions, i.e.~the first and the last site in the chain are considered as neighbors. The transition rate 
$A(\sigma, s)$ 
in the generator $A$ equals the rate constant of the respective reaction from $\sigma$ to $s$ (if there exists one). Off-diagonal elements which do not correspond to one of the possible reactions are set to zero and the diagonal elements $A(\sigma, \sigma)= - \sum_{s \neq \sigma} A(\sigma,  s)$ ensure conservation of probability. For a detailed discussion of the model, we refer to~\cite{GELSS2016,GELSS2017}. The generator represents a nearest neighbor interaction network. Thus, for the construction of the TT representation $\mathbf{A}$ we use the algorithm for the automatic generation of SLIM decompositions which was presented in~\cite{GELSS2017}. The inputs for the algorithm can be found in Appendix \ref{app: oxidation - SLIM}.

\begingroup
\renewcommand{\arraystretch}{1.3}
\begin{table}[t!]
\centering
\setlength\tabcolsep{1.5mm}
\begin{tabular}{llcccll}
\hline
\multicolumn{7}{l}{\textbf{Adsorption}} \\ \hline
$\textrm{R}^{\textrm{Ad}}_{\textrm{O}_2}$ &:&    $\varnothing_i + \varnothing_{j}$ & $\rightarrow$ & $\textrm{O}_i + \textrm{O}_{j}$ & , \quad $k^{\textrm{Ad}}_{\textrm{O}_2}$ &$=~~~1 $ \\
$\textrm{R}^{\textrm{Ad}}_{\textrm{CO}}$ &:&    $\varnothing_i$ & $\rightarrow$ & $\textrm{CO}_i$ & , \quad $k^{\textrm{Ad}}_{\textrm{CO}}$ &$=~~~0.01 $ \\
\hline
\multicolumn{7}{l}{\textbf{Desorption}} \\ \hline
$\textrm{R}^{\textrm{De}}_{\textrm{O}_2}$ &:&    $\textrm{O}_i+ \textrm{O}_{j}$ & $\rightarrow$ & $\varnothing_i + \varnothing_{j}$ & , \quad  $k^{\textrm{De}}_{\textrm{O}_2}$ &$=~~~0.001 $ \\
$\textrm{R}^{\textrm{De}}_{\textrm{CO}}$ &:&    $\textrm{CO}_i$ & $\rightarrow$ & $\varnothing_i$ & , \quad$k^{\textrm{De}}_{\textrm{CO}}$ &$=~~~1 $ \\
$\textrm{R}^{\textrm{De}}_{\textrm{CO}_2}$ &:&    $\textrm{CO}_i + \textrm{O}_{j} $ & $\rightarrow$ & $\varnothing_i + \varnothing_{j}$ & , \quad $k^{\textrm{De}}_{\textrm{CO}_2}$ &$=~~~1 $ \\
\hline
\multicolumn{7}{l}{\textbf{Diffusion}} \\ \hline
$\textrm{R}^{\textrm{Diff}}_{\textrm{O}}$ &:&    $\textrm{O}_i + \varnothing_{j}$ & $\rightarrow$ & $\varnothing_i + \textrm{O}_{j}$ & , \quad $k^{\textrm{Diff}}_{\textrm{O}}$ &$=~~~0.0001 $ \\
$\textrm{R}^{\textrm{Diff}}_{\textrm{CO}}$ &:&    $\textrm{CO}_i + \varnothing_{j}$ & $\rightarrow$ & $\varnothing_i + \textrm{CO}_{j}$ & , \quad $k^{\textrm{Diff}}_{\textrm{CO}}$ &$=~~~0.0001 $ \\
\hline
\end{tabular}
\caption{Elementary reaction steps and their corresponding rate constants (in s$^{-1}$): The reactions are defined on two neighboring sites $i$ and $j$, except for adsorption and desorption of CO, these reactions are defined only on site $i$.}
\label{table:reactions}
\end{table}
\endgroup

The complete list of possible reactions and corresponding rate constants is summarized in Table \ref{table:reactions}. Note that the employed rate constants differ from those used in \cite{GELSS2016,GELSS2017}. To our best knowledge, there exists no analytical solution for the stationary distribution of this irreversible Markov process. Thus, the rate constants have been chosen in such a way that we can compute fairly accurate kinetic Monte Carlo (kMC) estimates with adequate computational effort. Our aim is to benchmark the TT approach against kMC simulations of the stochastic process, for which we employ the {\sc kmos} package~\cite{HOFFMANN2014}. Unfortunately, we are not able to obtain accurate estimates for the stationary distribution from kMC simulations, simply because this would correspond to a sampling of up to quadrillions (if $d=32$) of expected values. For that reason, we focus on a few key observables, namely the turn-over frequencies (TOFs) of the adsorption and desorption reactions which are the average numbers of the different events per site and unit time. Since expected values are just scalar products between a given vector and the distribution, the error bounds obtained in Section~\ref{sec: odes} apply to expected values accordingly.

The Markovian master equation corresponding to the presented CO oxidation model is given by $\frac{\partial}{\partial t} \mathbf{p}(t) = \mathbf{A}^\top \mathbf{p}(t)$. Therefore, our results on approximability on ODEs from Section~\ref{sec: odes} in principle apply. In what follows, we will however consider the system when relaxed to the stationary distribution (for $t \to \infty$) and test the convergence of our TT approximations for different numbers $d$ of adsorption sites. For the stationary distribution $\mathbf{p}$, it holds that $\mathbf{A}^\top \mathbf{p} = \mathbf{0}$ and $\mathbf{p} \geq 0$ entry-wise, see~\cite{Keizer1972}. Therefore, we consider the eigenvalue problem
\begin{equation}\label{eq: oxidation - EVP}
\mathbf{A}^\top \mathbf{p} = \lambda \mathbf{p}
\end{equation}
with $\lambda=0$. It can be shown that the stationary distribution is indeed uniquely determined and $\lambda = 0$ is the eigenvalue with largest real part of $\mathbf{A}^\top$, see Appendix~\ref{app: I-A^T proof}. As $\mathbf{p}$ is the limit of ODE solutions, we expect low-rank approximability to hold and therefore try to approximate the eigenvector $\mathbf{p}$ in the low-rank TT format by employing ALS with rank bound $r$ as an eigenvalue solver for the eigenvalue of smallest modulus, see~\cite{HOLTZ2012}. Even though ALS was particularly developed for systems with symmetric (and positive definite) operators on the left hand side, we observed highly accurate approximations of the solutions for the non-symmetric problem given in \eqref{eq: oxidation - EVP}. For related studies dealing with non-symmetric systems we refer to \cite{DOLGOV2014, DOLGOV2015, GELSS2017, GELSS2017b}. We run the ALS scheme until we observe (almost) convergence of the approximate solutions and normalize the result such that $\lVert \mathbf{p} \rVert_1 = 1$. The kMC reference data as well as a brief methodological description can be found in Appendix~\ref{app: oxidation - kMC}.

\begin{figure}[t]
    \centering
    \begin{subfigure}[b]{0.45\textwidth}
        \centering
        \includegraphics[height=140px]{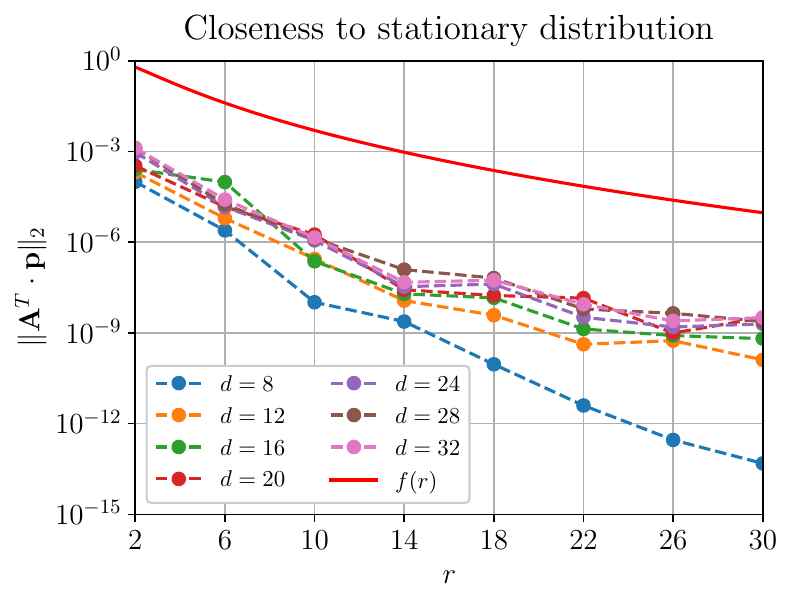}
        \caption*{~~(a)}
    \end{subfigure}
    \hfill
    \begin{subfigure}[b]{0.45\textwidth}
        \centering
        \includegraphics[height=140px]{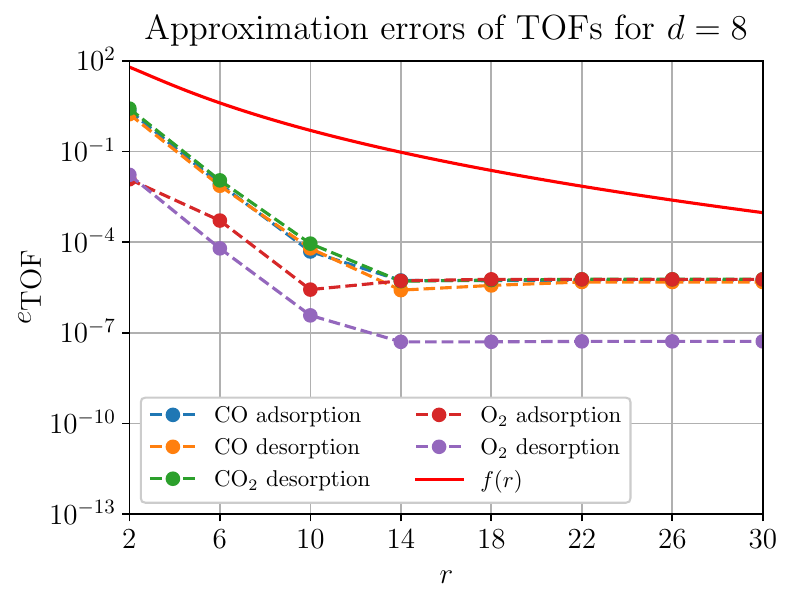}
        \caption*{~~(b)}
    \end{subfigure}
    \\[0.2cm]
    \begin{subfigure}[b]{0.45\textwidth}
        \centering
        \includegraphics[height=140px]{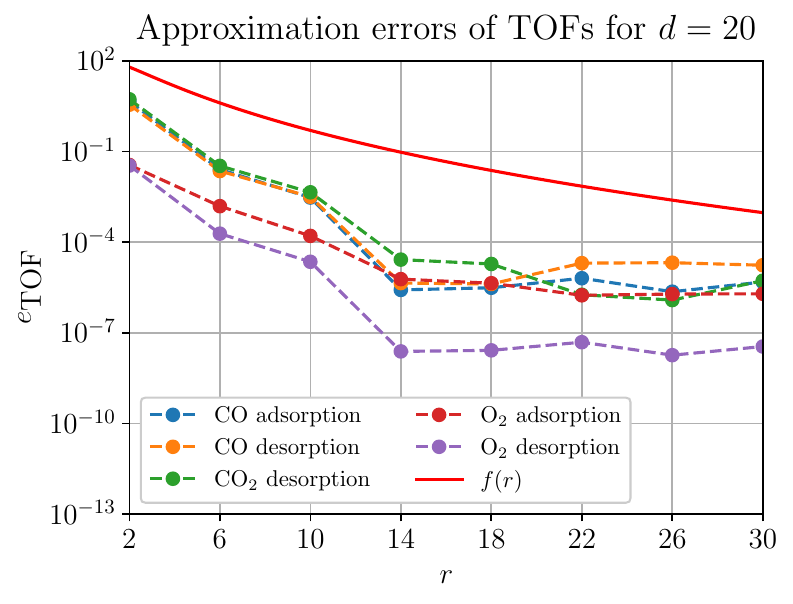}
        \caption*{~~(c)}
    \end{subfigure}
    \hfill
    \begin{subfigure}[b]{0.45\textwidth}
        \centering
        \includegraphics[height=140px]{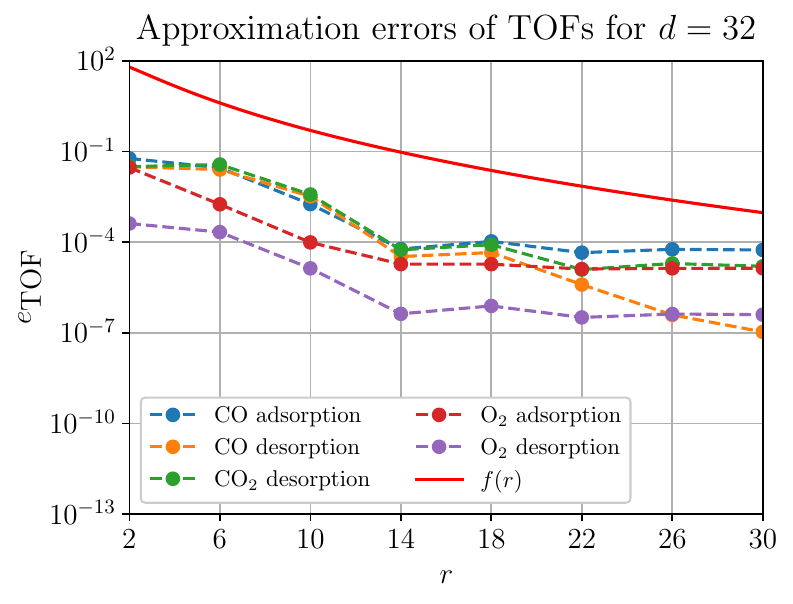}
        \caption*{~~(d)}
    \end{subfigure}
    \caption{Simulation of the CO oxidation: (a)  Norm of the derivatives $\mathbf{A}^\top \cdot \mathbf{p}$ for different dimensions and TT rank bounds. (b)-(d) Relative errors of the approximated TOFs for dimensions $d=8,20,32$ and different TT ranks. To illustrate consistency with the theoretical results, we also plot the functions $f(r) = c \cdot \mathrm{e}^{-\ln^2(r)}$ with (a) $c=1$ and (b)-(d) $c=100$, respectively.}
    \label{fig: TOFs}
\end{figure}

As a measure for the deviation of our approximative solutions $\mathbf{p}$ from the true stationary state we consider, in dependence of the prescribed rank limit $r$ and for different dimensions $d \in [8,32]$, the Euclidean norm of $\mathbf{A}^\top \mathbf{p}$, which can be interpreted as time derivative of the time-dependent process going through the state $\mathbf{p}$. The results are shown in Figure~\ref{fig: TOFs} (a). Note that the largest system we consider has $3^{32} > 10^{15}$ states, which makes classical numerical methods for solving systems of linear equations or eigenvalue problems infeasible. As one can see, the error bounds described in Section~\ref{sec: odes} apply to the time derivatives,~too.

In Figure \ref{fig: TOFs} (b)--(d), we display the relative difference $e_\textrm{TOF}$ between the referential kMC estimates and the TT approximations. The curves show the expected behavior until a certain point. From then on, higher ranks seem to have no further positive effect. This occurs when the TT approximations reach a comparable (or even better) accuracy as the kMC results. For these, we found relative accuracies in the range $[10^{-7}, 10^{-4}]$, see Appendix~\ref{app: oxidation - kMC}.

\appendix

\section{Proof of Theorem~\ref{th: result by Arad et al}}\label{appendix: proof of theorem}

The proof is adapted from~\cite{Aradetal2013}. We first present two lemmas. The first one is a special case on a more general result on how to compare ranks of different tensor flattenings~\cite{CarliniKleppe}. 

\begin{lemma}\label{lem: shiFdting rank}
Assume $d \ge 3$ and $n_\mu \le n$ for $\mu=1,\dots,d$. Let $\bu \in \bV$ and $1\le \mu,\nu \le d-1$ be given. Then
\[
\rank_\mu(\bu) \le n^{\abs{\mu-\nu}} \rank_\nu(\bu).
\]
\end{lemma}

\begin{proof}
We prove this for $\nu = \mu + 1$. The proof for $\nu = \mu-1$ is similar. All other cases follow by induction. By definition, we can write
\[
 \bu = \sum_{j = 1}^{\rank_{\nu}(\bu)} v_j \otimes w_j 
\]
where $v_j \in (V_1 \otimes \dots \otimes V_\mu) \otimes V_{\nu}$ and $w_j \in V_{\nu+1} \otimes \dots \otimes V_d$. The brackets indicate that we can regard the~$v_j$ as second order tensors. Hence, since $\dim(V_{\nu}) \le n$, there exist $e_1,\dots,e_n \in V_{\nu}$ such that every $v_j$ can be written as
\[
 v_j = \sum_{i=1}^n \hat{v}_{ij} \otimes e_i
\]
with $\hat{v}_{ij} \in V_1 \otimes \dots \otimes V_{\mu}$. As a result,
\[
 \bu =  \sum_{j = 1}^{\rank_{\nu}(\bu)} \left( \sum_{i=1}^n \hat{v}_{ij} \otimes e_i \right) \otimes w_j = \sum_{i=1}^n \sum_{j = 1}^{\rank_{\nu}(\bu)} \hat{v}_{ij} \otimes (e_i  \otimes w_j),
\]
which shows $\rank_{\mu}(\bu) \le n \rank_{\nu}(\bu)$.
\end{proof}

The second lemma concerns the existence of unisolvent interpolation points for subspaces of multivariate polynomials. Note that we only require existence of such points, which follows from more or less elementary linear algebra.

\begin{lemma}\label{lem: interpolation problem}
Let $P_\ell$ denote the subspace of univariate $\K$-polynomials of degrees at most $\ell$. Let $V$ be an $M$-dimensional subspace of $\bigotimes_{i=1}^s P_\ell$. Then there exist $M$ distinct points $\bz_1, \dots, \bz_M \in \K^{s}$ such that the point evaluations $\delta_{\bz_j} \vcentcolon p \mapsto p(\bz_j)$, $j=1,\dots,M$, are linearly independent in $V^*$.
\end{lemma}

\begin{proof}
Let $P = \bigotimes_{i=1}^s P_\ell$ for convenience. Fix $\varphi \in V^*$ and let $\Pi$ be a projection from $P$ on $V$. Then $\varphi \circ \Pi$ is a linear functional on $P$. For the space $P$ it is well known that there exist points $\mathbf{y}_1,\dots,\mathbf{y}_{(\ell + 1)s}$ such that the interpolation problem is uniquely solvable (consider, for instance, tensor products of Lagrange basis polynomials on a tensor product grid). In other words, the point evaluations $\delta_{\mathbf{y}_j}$, $j=1,\dots,(\ell+1)s$, form a basis of $P^*$. Hence, there exist $a_1,\dots,a_{(\ell+1)s} \in \K$ such that
\[
\varphi(p) = (\varphi \circ \Pi)( p) = \sum_{j=1}^{(\ell+1)s} a_j \delta_{\mathbf{y}_j}(p)
\]
for every $p \in V$. This shows that the $\delta_{\mathbf{y}_j}$ (more precisely, their restrictions to $V$) are a generating system for $V^*$, and hence contain a basis of $V^*$.
\end{proof}

\subsubsection*{Proof of Theorem~\ref{th: result by Arad et al}}

Let $\bu \in \bV$, $\bu \neq 0$. Fix $1 \le \mu \le d-1$, $s \ge 2$ and $\ell \ge 1$. When $s \ge d$, it trivially holds
\[
\max_\mu \rank_\mu(\bA^\ell \bu) \le n^{d/2} \le \frac{n^s}{n-1} \max_\mu \rank_\mu(\bu)
\]
(recall $d \ge 3$ and $n \ge 2$). The upper bound asserted by~\eqref{eq: detailed rank estimate} is even larger, so this case is clear.

Let us therefore assume in the following that $s \le d-1$. Fix an arbitrary $1 \le \mu \le d-1$. Our goal is to investigate how the application of $\bA^\ell$ to $\bu$ affects the $\mu$-rank for that particular $\mu$. For this, we first rewrite $\bA$. There exists $\hat{\mu} \ge 0$ such that
\begin{equation}\label{eq: mu window}
\hat{\mu} < \mu \le \hat{\mu}  + s \le d-1.
\end{equation}
We can then write
\[
\bA = \sum_{i=1}^s \hbA_i
\]
where $\hbA_1 = \bA_{< \hat{\mu}+2}$, $\hbA_s = \bA_{> \hat{\mu} + s - 1}$ and $\hbA_i = \bA_{\hat{\mu} + i}$ for $i = 2,\dots,s-1$. Introducing multi-indices $\alpha \in \{ 0,1,\dots,\ell\}^s$, we denote by $\bG_\alpha$ the sum of all possible products of $\hbA_{1}, \dots, \hbA_{s}$ of length $\abs{\alpha}_1 = \alpha_1 + \dots + \alpha_s$ in which for every $j = 1,\dots,s$ the operator~$\hbA_j$ appears exactly~$\alpha_j$ times (at arbitrary positions). Then we can write $\bA^\ell$ as
\[
 \bA^\ell = \left( \sum_{i=1}^s \hbA_i \right)^\ell = \sum_{\abs{\alpha}_1 = \ell} \bG_\alpha.
\]
Every multi-index $\alpha$ in the above summation has at least one entry $\alpha_i$ satisfying $\alpha_i \le \floor{\ell/s}$. Therefore, there exists an (in general non-unique) decomposition of the above sum into
\begin{equation}\label{eq: decomposition 1}
\bA^\ell = \sum_{i=1}^s \sum_{k = 0}^{\floor{\ell/s}} \sum_{\substack{\abs{\alpha}_1 = \ell \\ \alpha_i = k}} c_\alpha \bG_\alpha,
\end{equation}
where the $c_\alpha$ are either one or zero (to avoid potential multiple appearances of equal $\bG_\alpha$).

Now consider, for fixed $i$ and $k$ in the summation ranges of~\eqref{eq: decomposition 1}, the subspace $V$ of homogeneous multivariate algebraic polynomials spanned by monomials $\bx^\alpha = x_1^{\alpha_1} \cdots x_s^{\alpha_s}$ of the same degree $\abs{\alpha}_1 = \ell$ with fixed $\alpha_i = k$ (to be clear: we set $x_i^{\alpha_i} = 1$ for $\alpha_i = 0$). This space has the same dimension as the space of all homogeneous polynomials in $s-1$ variables of degree $\ell-k$ which is known to equal
\begin{equation}\label{eq: dimension M}
M = \binom{\ell - k + s-2}{s-2},
\end{equation}
and the corresponding monomials $\bx^\alpha$ form a basis of $V$. By Lemma~\ref{lem: interpolation problem}, there exist points $\bz_\beta = (z_{\beta,1},\dots,z_{\beta,s}) \in \K^s$ such that the ``Vandermonde'' matrix ${[\bz_\beta^\alpha]}_{\alpha,\beta} = {[z_{\beta,1}^{\alpha_1} \cdots z_{\beta,s}^{\alpha_s}]}_{\beta,\alpha} \in \K^{M \times M}$, which arises from evaluating the basis monomials $\bx^\alpha$ in the points $\bz_\beta$, is invertible. Here $\beta$ runs through the same set of multi-indices as $\alpha$, i.e., $\beta \in  \{0, 1,\dots,\ell\}^s$, $\abs{\beta}_1 = \ell$, $\beta_i = k$. Since the matrix ${[\bz_\beta^\alpha]}_{\alpha,\beta}$ is invertible, there exist coefficients $b_\beta$ such that
\[
\sum_{\substack{\abs{\beta}_1 = \ell \\ \beta_i = k}} b_\beta^{} \bz_\beta^\alpha = c_\alpha
\]
for all $\alpha$ with $\abs{\alpha}_1 = \ell$ and $\alpha_i = k$, where $c_\alpha$ are the coefficients in~\eqref{eq: decomposition 1}. Then~\eqref{eq: decomposition 1} becomes
\[
\bA^\ell = \sum_{i=1}^s \sum_{k = 0}^{\floor{\ell/s}} \sum_{\substack{\abs{\beta}_1 = \ell \\ \beta_i = k}} b_\beta \sum_{\substack{\abs{\alpha}_1 = \ell \\ \alpha_i = k}} \bz_\beta^\alpha \bG_\alpha,
\]
and so, using that $\mu$-rank is subadditive and invariant under scalar multiplication, we get
\[
\rank_\mu(\bA^\ell \bu) \le \sum_{i=1}^s \sum_{k = 0}^{\floor{\ell/s}} \sum_{\substack{\abs{\beta}_1 = \ell \\ \beta_i = k}} \rank_\mu\left( \sum_{\substack{\abs{\alpha}_1 = \ell \\ \alpha_i = k}} \bz_\beta^\alpha \bG_\alpha \bu \right).
\]
Moreover, by invoking Lemma~\ref{lem: shiFdting rank}, we can estimate the $\mu$-rank by the ``nearby'' $(\hat{\mu} + i)$-ranks:
\begin{equation}\label{eq: preliminary estimate}
\rank_\mu(\bA^\ell \bu) \le \sum_{i=1}^s \sum_{k = 0}^{\floor{\ell/s}} \sum_{\substack{\abs{\beta}_1 = \ell \\ \beta_i = k}} n^{\abs{\mu - (\hat{\mu} +  i)}} \rank_{\hat{\mu} + i}\left( \sum_{\substack{\abs{\alpha}_1 = \ell \\ \alpha_i = k}} \bz_\beta^\alpha \bG_\alpha \bu \right).
\end{equation}

Let us now investigate the terms in brackets in~\eqref{eq: preliminary estimate}. Recall that $\bG_\alpha$ is (for $\abs{\alpha}_1 = \ell$) the sum of all monomials $\hA_{i_1} \cdots \hA_{i_\ell}$ in which every $\hA_j$ appears $\alpha_j$ times. In particular, if $\hA_j$ appears in such a monomial, then $\alpha_j > 0$. We can therefore redistribute the factors of the attached scalar monomial $\bz_\beta^\alpha = \prod_{j=1}^s z_{\beta,j}^{\alpha_j}$ accordingly, that is,
\[
\bz_\beta^\alpha \hA_{i_1} \cdots \hA_{i_\ell} = (z_{\beta,i_1} \hA_{i_1}) \cdots (z_{\beta,i_\ell} \hA_{i_\ell}).
\]
Substituting $\barA_j = z_{\beta,j} \hA_j$ (we consider $\beta$ fixed and omit dependence on $\beta$), it then follows that $\sum_{\substack{\abs{\alpha}_1 = \ell \\ \alpha_i = k}} \bz_\beta^\alpha \bG_\alpha$ is the sum of all possible monomials $\barA_{i_1} \cdots \barA_{i_\ell}$ of length $\ell$ in which $\barA_i$ appears precisely $\alpha_i = k$ times. The same operator can be obtained by fully expanding $(\barA_1 + \dots + \barA_s)^\ell = (\barA_{< i} + \barA_i + \barA_{> i})^\ell$, where $\barA_{< i} = \sum_{j < i} \barA_j$, $\barA_{> i} = \sum_{j > i}  \barA_j$, into monomials of only three operators $\barA_{< i}$, $\barA_i$, and $\barA_{> i}$, and collecting the ones in which $\barA_i$ appears $k$ times. However, since $\barA_{< i}$ and $\barA_{> i}$ commute, many of these monomials result in the same operator. Repeating the logic at the beginning of Section~\ref{sec: rank-increase}, it is therefore enough to expand into `partially ordered' monomials as gathered in the set $S_{k,\ell}$ defined in~\eqref{eq: set Skl}. Each of these monomials increases the $(\hat{\mu} + i)$-rank by a factor at most $R^k$, since only $\barA_i = z_{\beta,i} \hbA_i = z_{\beta,i} \bA_{\hat{\mu} + i}$ (which appears $k$ times) increases the $(\hat{\mu} + i)$-rank and $\RI(\bA_{\hat{\mu} + i}) \le R$ by assumption. In summary, using~\eqref{abs Skl}, we estimate
\begin{equation}\label{eq: partial estimate 1}
\rank_{\hat{\mu} + i}\left( \sum_{\substack{\abs{\alpha} = \ell \\ \alpha_i = k}} \bz_\beta^\alpha \bG_\alpha \bu \right) \le \abs{S_{k,\ell}} \cdot R^k \cdot \rank_{\hat{\mu} + i}(\bu) \le \binom{\ell + k + 1}{2k+1} \cdot R^k \max_\mu \rank_\mu(\bu).
\end{equation}

In the estimate~\eqref{eq: preliminary estimate}, the multi-index $\beta$ runs through $M$ different values with $M$ given by~\eqref{eq: dimension M}. Hence, since the right-hand side of~\eqref{eq: partial estimate 1} is independent of $\beta$ and $i$,~\eqref{eq: preliminary estimate} can be simplified to 
\begin{equation}\label{eq: partial estimate 2}
\rank_\mu(\bA^\ell \bu) \le \left(\sum_{i=1}^s n^{\abs{\mu - \hat{\mu} -  i}} \right) \left(\sum_{k = 0}^{\floor{\ell/s}} \binom{\ell - k + s-2}{s-2} \binom{\ell + k + 1}{2k+1} \cdot R^k \right) \max_\mu \rank_\mu(\bu).
\end{equation}
In the remaining part of the proof, we estimate the single parts in this expression separately.

By~\eqref{eq: mu window} we have that $1 \le \mu - \hat \mu \le s$. This implies (recall $n \ge 2$):
\begin{equation}\label{p1}
\sum_{i=1}^s n^{\abs{\mu - \hat{\mu} -  i}} \le \sum_{i=1}^s n^{i-1} = \frac{n^s -1 }{n-1} \le \frac{n^s}{n-1}.
\end{equation}

For the binomial coefficients in~\eqref{eq: partial estimate 2} we use the estimates
\begin{equation}\label{eq: binomial estimates}
\binom{a + b}{a} \le \frac{(a+b)^{a+b}}{(a+1)^a b^b} = \left( \frac{a+b}{a+1} \right)^a \left( 1 + \frac{a}{b} \right)^b \le \left( \frac{a+b}{a+1} \right)^a \mathrm{e}^a  \le \left( 1 + \frac{b}{a} \right)^a \mathrm{e}^a,
\end{equation}
which holds for integers $a \ge 0$ and $b \ge 1$, see~\cite[\S3.1.30]{Mitrinovic1970} and~\cite[Lemma~3.4]{KuehnSickelUllrich2014} for a proof of the first inequality. Hence
\begin{equation}\label{p2}
\binom{\ell - k + s-2}{s-2} \le \left( 1+  \frac{\ell - k}{s-2} \right)^{s-2} \mathrm{e}^{s-2} \le \left( 1 + \frac{\ell}{s} \right)^{s} \mathrm{e}^{s-2},
\end{equation}
where we have used that the function $t \mapsto (1 + \ell/t)^t$ is monotonically increasing for $t > 0$. Further, it holds that (using only the second inequality in~\eqref{eq: binomial estimates})
\begin{equation}\label{p3}
\binom{\ell + k + 1}{2k+1}
\le
\left( \frac{\ell + k + 1}{2k+2} \right)^{2k + 1} \mathrm{e}^{2k+1} = \left( 1 + \frac{\ell}{k + 1} \right)^{2k + 1} \left(\frac{\mathrm{e}}{2}\right)^{2k+1}.
\end{equation}
The function $t \mapsto (1 + \ell /t)^{2t - 1} = (1 + \ell /t)^{2t} (1 + \ell/t)^{-1}$ is also monotonically increasing for $t > 0$. Therefore, for $k = 0,1,\dots,\floor{\ell /s}$ we have
\[
\left( 1 + \frac{\ell}{k + 1} \right)^{2k + 1} \le \left( 1 + \frac{\ell}{\ell/s + 1} \right)^{2 \ell /s + 1} \le (1 + s)^{2 \ell/s + 1}.
\]
For such $k$ an upper bound for~\eqref{p3} hence is
\begin{equation}\label{p4}
\binom{\ell + k + 1}{2k+1} \le \frac{\mathrm{e}}{2} (1 + s)^{2 \ell/s + 1} \left(\frac{\mathrm{e}^2}{4}\right)^{k}.
\end{equation}

Inserting~\eqref{p1},~\eqref{p2} and~\eqref{p4} into~\eqref{eq: partial estimate 2}, we arrive at
\[
\rank_\mu(\bA^\ell \bu) \le \frac{n^s}{2(n-1)} \left( 1 + \frac{\ell}{s} \right)^{s} \mathrm{e}^{s-1} (1 + s)^{2 \ell/s + 1}   \left(\sum_{k = 0}^{\floor{\ell/s}}  \left(\frac{\mathrm{e}^2 R}{4}\right)^{k} \right) \max_\mu \rank_\mu(\bu).
\]
Since $e^2R/4 - 1 \ge eR/4$ for $R \ge 1$, we can estimate
\[
 \sum_{k = 0}^{\floor{\ell/s}}  \left(\frac{\mathrm{e}^2 R}{4}\right)^{k} = \frac{\left(\frac{\mathrm{e}^2 R}{4}\right)^{\floor{\ell/s}+1} - 1}{\frac{\mathrm{e}^2 R}{4} - 1} \le \mathrm{e} \left(\frac{\mathrm{e}^2 R}{4}\right)^{\ell/s},
\]
which then yields the asserted inequality. \hfill \qed

\section{Ising model}

\subsection{Tensor decomposition of the stationary distribution}\label{app: construction of pi}

An exact TT representation of the stationary distribution $\pi$ in~\eqref{eq: Ising - configuration probability} can be derived from
\begingroup
\setlength{\arraycolsep}{1pt}
\begin{equation}
\label{eq: stat dist tt}
    \mathbf{\Pi} = \core{ \begin{pmatrix} \mathrm e^\beta \\ 0 \end{pmatrix} & \begin{pmatrix} 0 \\ \mathrm e^{-\beta} \end{pmatrix}} \otimes
    \core{ \begin{pmatrix} \mathrm e^{2\beta} \\ 0 \end{pmatrix} & \begin{pmatrix} 0 \\ \mathrm e^{-2\beta} \end{pmatrix} \\[0.4cm] \begin{pmatrix} 1 \\ 0 \end{pmatrix} & \begin{pmatrix} 0 \\ 1 \end{pmatrix}} \otimes \ldots \otimes
    \core{ \begin{pmatrix} \mathrm e^{2\beta} \\ 0 \end{pmatrix} & \begin{pmatrix} 0 \\ \mathrm e^{-2\beta} \end{pmatrix} \\[0.4cm] \begin{pmatrix} 1 \\ 0 \end{pmatrix} & \begin{pmatrix} 0 \\ 1 \end{pmatrix}}
    \otimes
    \core{ \begin{pmatrix} \mathrm e^{2\beta} \\ \mathrm e^{-2\beta} \end{pmatrix}  \\[0.4cm]  \begin{pmatrix} 1 \\ 1 \end{pmatrix}}.
\end{equation}
\endgroup
Here, we use the core notation (also called strong Kronecker product notation) of the TT format, cf.~\cite{Kazeev2012, Keller2015, GELSS2017}, where in our case the $2 \times \dots \times 2$ tensor $\mathbf{\Pi}$ is written as a product of two-dimensional arrays (cores) containing vectors as elements. The evaluation of this product follows the usual rules of matrix multiplication but using the Kronecker product for multiplying entries. Observe that each core has two matrix slices and an entry $\mathbf{\Pi}(\sigma_1,\dots,\sigma_d)$, $\sigma_i \in \{\pm 1\}$, is obtained by multiplying corresponding slices in the usual way according to~\eqref{eq: TT format}. Note that the decomposition~\eqref{eq: stat dist tt} has to be normalized after construction in order to obtain a TT representation of $\pi$.

To show the correctness of the above expression, we use induction over $d$: For $d=2$, we have
\begin{equation*}
\mathbf{\Pi} = \core{ \begin{pmatrix} \mathrm e^\beta \\ 0 \end{pmatrix} & \begin{pmatrix} 0 \\ \mathrm e^{-\beta} \end{pmatrix}} \otimes \core{ \begin{pmatrix} \mathrm e^{2\beta} \\ \mathrm e^{-2\beta} \end{pmatrix}  \\[0.4cm]  \begin{pmatrix} 1 \\ 1 \end{pmatrix}} \widehat{=} \begin{pmatrix}
\mathrm e^{3 \beta} & \mathrm e^{- \beta} \\ \mathrm e^{- \beta} & \mathrm e^{-\beta}
\end{pmatrix}
\end{equation*}
which by~\eqref{eq: Ising - configuration probability} agrees with $\pi(\sigma)$ up to scaling. For the induction step $d \rightarrow d+1$, we assume the TT decomposition of $\mathbf{\Pi}$ for $d$ is given by~\eqref{eq: stat dist tt}. If we set all entries with $\sigma_d = -1$ to zero, we obtain the tensor
\begin{equation*}
\mathbf{\Pi}^+ = \core{ \begin{pmatrix} \mathrm e^\beta \\ 0 \end{pmatrix} & \begin{pmatrix} 0 \\ \mathrm e^{-\beta} \end{pmatrix}} \otimes
\core{ \begin{pmatrix} \mathrm e^{2\beta} \\ 0 \end{pmatrix} & \begin{pmatrix} 0 \\ \mathrm e^{-2\beta} \end{pmatrix} \\[0.4cm] \begin{pmatrix} 1 \\ 0 \end{pmatrix} & \begin{pmatrix} 0 \\ 1 \end{pmatrix}} \otimes \ldots \otimes
\core{ \begin{pmatrix} \mathrm e^{2\beta} \\ 0 \end{pmatrix} & \begin{pmatrix} 0 \\ \mathrm e^{-2\beta} \end{pmatrix} \\[0.4cm] \begin{pmatrix} 1 \\ 0 \end{pmatrix} & \begin{pmatrix} 0 \\ 1 \end{pmatrix}}
\otimes
\core{ \begin{pmatrix} \mathrm e^{2\beta} \\ 0 \end{pmatrix}  \\[0.4cm]  \begin{pmatrix} 1 \\ 0 \end{pmatrix}}.
\end{equation*}
Likewise, we obtain 
\begin{equation*}
\mathbf{\Pi}^-  = \core{ \begin{pmatrix} \mathrm e^\beta \\ 0 \end{pmatrix} & \begin{pmatrix} 0 \\ \mathrm e^{-\beta} \end{pmatrix}} \otimes
\core{ \begin{pmatrix} \mathrm e^{2\beta} \\ 0 \end{pmatrix} & \begin{pmatrix} 0 \\ \mathrm e^{-2\beta} \end{pmatrix} \\[0.4cm] \begin{pmatrix} 1 \\ 0 \end{pmatrix} & \begin{pmatrix} 0 \\ 1 \end{pmatrix}} \otimes \ldots \otimes
\core{ \begin{pmatrix} \mathrm e^{2\beta} \\ 0 \end{pmatrix} & \begin{pmatrix} 0 \\ \mathrm e^{-2\beta} \end{pmatrix} \\[0.4cm] \begin{pmatrix} 1 \\ 0 \end{pmatrix} & \begin{pmatrix} 0 \\ 1 \end{pmatrix}}
\otimes
\core{ \begin{pmatrix} 0 \\ \mathrm e^{-2 \beta} \end{pmatrix}  \\[0.4cm]  \begin{pmatrix} 0 \\ 1  \end{pmatrix}}
\end{equation*}
when we set all entries with $\sigma_d=+1$ to zero. To proceed to $d+1$, in the case of $\sigma_d = +1$, we need to multiply all values in $\mathbf{\Pi}^+$ by $e^{\beta \sigma_d \sigma_{d+1} + \beta \sigma_{d+1}} = e^{2 \beta \sigma_{d+1}}$, and in the case of $\sigma_d = -1$, we multiply all values in $\mathbf{\Pi}^-$ by $1$. This implies the representation
\begin{equation}\label{eq: stat dist ind}
\core{\mathbf{\Pi}^+ & \mathbf{\Pi}^-} \otimes
\core{\begin{pmatrix} \mathrm e^{2 \beta} \\ \mathrm e ^{-2\beta}\end{pmatrix} \\ \begin{pmatrix}1 \\ 1\end{pmatrix}}
\end{equation}
for the (unnormalized) stationary distribution corresponding to the order $d+1$. This proves the correctness of the TT representation of the stationary distribution $\pi$ since~\eqref{eq: stat dist ind} can be brought into the same form as~\eqref{eq: stat dist tt} due to the shared cores of $\mathbf{\Pi}^+$ and $\mathbf{\Pi}^-$.

\subsection{Tensor decomposition of the infinitesimal generator}\label{app: construction of W}

\noindent First, given a configuration vector $\sigma =(\sigma_1 , \dots , \sigma_d)^\top \in \mathcal{S}$ we define the vector $\tilde\sigma$ by
\begin{equation*}
\tilde\sigma_i = 
\begin{cases}
1, & \textrm{if}~~\sigma_i = +1 \\
2, & \textrm{if}~~\sigma_i = -1
\end{cases}
\end{equation*}
which will be more convenient to work with. For the construction of a TT operator $\mathbf{A}$ with
\begin{equation*}
\mathbf{A}_{ \tilde\sigma_1 , \tilde\nu_1 , \dots, \tilde\sigma_d, \tilde\nu_d} = A(\sigma ,\nu)
\end{equation*}
we consider the transition rate for a pair of configurations $(\sigma, \nu)$ with $\sigma_k = -\nu_k$ for some $1 \le k \le d$ and $\sigma_i = \nu_i$ for $i \neq k$. In the case $k \in \{2, \dots , d-1 \}$ we obtain from~\eqref{eq: Ising - transition rates} and~\eqref{eq: Ising - Hamiltonian function} that
\begin{equation*}
A(\sigma, \nu) = e^{-\frac{\beta}{2}(H(\nu) - H(\sigma))}=e^{-\beta \sigma_{k-1} \sigma_k} \cdot e^{-\beta  \sigma_k} \cdot e^{-\beta \sigma_k \sigma_{k+1}},
\end{equation*}
whereas in the boundary cases $k=1$ and $k=d$, we have
\begin{equation*}
A(\sigma, \nu) = e^{-\beta \sigma_1} \cdot e^{-\beta \sigma_1 \sigma_2 }  
\end{equation*}
and 
\begin{equation*}
A(\sigma , \nu) = e^{-\beta \sigma_{d-1} \sigma_d}  \cdot e^{-\beta \sigma_d},
\end{equation*}
respectively. Together with $A(\sigma, \sigma) = - \sum_{\nu \neq \sigma} A(\sigma, \nu)$ this leads to a canonical tensor representation of the form
\begin{equation*}
\begin{split}
\mathbf{A} 	= & M_1 \otimes L_1 \otimes I \otimes \dots \otimes I + M_2 \otimes L_2 \otimes I \otimes \dots \otimes I  \\
& + L_1 \otimes M_1 \otimes L_1 \otimes I \otimes \dots \otimes I + L_2 \otimes M_2 \otimes L_2 \otimes I \otimes \dots \otimes I\\
& + \dots {} \\
& + I \otimes \dots \otimes I \otimes L_1 \otimes M_1 \otimes L_1 + I \otimes \dots \otimes I \otimes L_2 \otimes M_2 \otimes L_2\\
& + I \otimes \dots \otimes I \otimes L_1 \otimes M_1 + I \otimes \dots \otimes I \otimes L_2\otimes M_2,
\end{split}
\end{equation*}
where $I$ is the identity matrix in $\mathbb{R}^2$ and
\begin{align*}
L_1 &= \begin{pmatrix}e^{- \beta} & 0 \\ 0 & e^{\beta} \end{pmatrix},     &  L_2 &= L_1^{-1} = \begin{pmatrix}e^{\beta} & 0 \\ 0 & e^{-\beta} \end{pmatrix},\\ 
M_1 &= \begin{pmatrix} -e^{-\beta} & e^{-\beta} \\ 0 & 0 \end{pmatrix},&  M_2 &= \begin{pmatrix}0 & 0 \\ e^{\beta} & - e^{\beta}  \end{pmatrix}.
\end{align*}
In order to convert this system into a nearest neighbor interaction network, we contract pairs of neighboring cores to so-called supercores such that at most two supercores are unequal to $I \otimes I$. Hence, assuming $d$ is even, the resulting operator acts on $4 \times \dots \times 4$ tensors of order $d/2$ and can be expressed in the core notation for TT operators (entries of arrays are $4 \times 4$ matrices) as

\begingroup 
\setlength\arraycolsep{2.1pt}
\begin{equation*}
\begin{split}
    \mathbf{A} = &
    \core{
        \displaystyle \sum_{k=1}^2 \makebox[1.5em][c]{$M_k$} \otimes \makebox[1.5em][c]{$L_k$}  & \makebox[1.5em][c]{$L_1$} \otimes \makebox[1.5em][c]{$M_1$} & \makebox[1.5em][c]{$L_2$} \otimes \makebox[1.5em][c]{$M_2$} & \makebox[1.5em][c]{$I$} \otimes \makebox[1.5em][c]{$L_1$} & \makebox[1.5em][c]{$I$} \otimes \makebox[1.5em][c]{$L_2$} & \makebox[1.5em][c]{$I$} \otimes \makebox[1.5em][c]{$I$}
    }\\
    & ~~\otimes 
    \core{
	\makebox[1.5em][c]{$I$} \otimes \makebox[1.5em][c]{$I$}    & 0               & 0               & 0             & 0             & 0             \\
	\makebox[1.5em][c]{$L_1$} \otimes \makebox[1.5em][c]{$I$}  & 0               & 0               & 0             & 0             & 0             \\
	\makebox[1.5em][c]{$L_2$} \otimes \makebox[1.5em][c]{$I$}  & 0               & 0               & 0             & 0             & 0             \\
	\makebox[1.5em][c]{$M_1$} \otimes \makebox[1.5em][c]{$L_1$} & 0               & 0               & 0             & 0             & 0             \\
	\makebox[1.5em][c]{$M_2$} \otimes \makebox[1.5em][c]{$L_2$} & 0               & 0               & 0             & 0             & 0             \\
        0               & \makebox[1.5em][c]{$L_1$} \otimes \makebox[1.5em][c]{$M_1$} & \makebox[1.5em][c]{$L_2$} \otimes \makebox[1.5em][c]{$M_2$} & \makebox[1.5em][c]{$I$} \otimes \makebox[1.5em][c]{$L_1$} & \makebox[1.5em][c]{$I$} \otimes \makebox[1.5em][c]{$L_2$} & \makebox[1.5em][c]{$I$} \otimes \makebox[1.5em][c]{$I$}
    }\\
    & ~~\otimes ~\dots \\
    & ~~\otimes 
    \core{
	\makebox[1.5em][c]{$I$} \otimes \makebox[1.5em][c]{$I$}    & 0               & 0               & 0             & 0             & 0             \\
	\makebox[1.5em][c]{$L_1$} \otimes \makebox[1.5em][c]{$I$}  & 0               & 0               & 0             & 0             & 0             \\
	\makebox[1.5em][c]{$L_2$} \otimes \makebox[1.5em][c]{$I$}  & 0               & 0               & 0             & 0             & 0             \\
	\makebox[1.5em][c]{$M_1$} \otimes \makebox[1.5em][c]{$L_1$} & 0               & 0               & 0             & 0             & 0             \\
	\makebox[1.5em][c]{$M_2$} \otimes \makebox[1.5em][c]{$L_2$} & 0               & 0               & 0             & 0             & 0             \\
        0               & \makebox[1.5em][c]{$L_1$} \otimes \makebox[1.5em][c]{$M_1$} & \makebox[1.5em][c]{$L_2$} \otimes \makebox[1.5em][c]{$M_2$} & \makebox[1.5em][c]{$I$} \otimes \makebox[1.5em][c]{$L_1$} & \makebox[1.5em][c]{$I$} \otimes \makebox[1.5em][c]{$L_2$} & \makebox[1.5em][c]{$I$} \otimes \makebox[1.5em][c]{$I$}
    }
    \otimes 
    \core{
	\makebox[1.5em][c]{$I$} \otimes \makebox[1.5em][c]{$I$}     \\
	\makebox[1.5em][c]{$L_1$} \otimes \makebox[1.5em][c]{$I$}  \\
	\makebox[1.5em][c]{$L_2$} \otimes \makebox[1.5em][c]{$I$}   \\
	\makebox[1.5em][c]{$M_1$} \otimes \makebox[1.5em][c]{$L_1$} \\
	\makebox[1.5em][c]{$M_2$} \otimes \makebox[1.5em][c]{$L_2$} \\
        \displaystyle \sum_{k=1}^2 \makebox[1.5em][c]{$L_k$} \otimes \makebox[1.5em][c]{$M_k$}              
    },
    \end{split}
\end{equation*}
\endgroup
which resembles the structure of a SLIM decomposition, see \cite{GELSS2017}.

\section{CO oxidation}

\subsection{Input for SLIM decomposition}\label{app: oxidation - SLIM}

For constructing the operator corresponding to the MME, we use Algorithm 2 from \cite{GELSS2017} with the inputs
\newcommand\z[1]{\makebox[\widthof{$\begin{pmatrix} k_{\textrm{O}_2}^{\textrm{Ad}} & 0 & 0 \\ \w0 & \w0 & \w0 \\ 0 & 0 & 0 \end{pmatrix}$}]{$#1$}}
\newcommand\w[1]{\makebox[0.8cm]{$#1$}}
\begin{center}
$\begin{array}{ccccccc}
    \mathrm{a}_{\mu,1}     & = & \z{\begin{pmatrix} \w{k_{\textrm{CO}}^{\textrm{Ad}}} & \w0 & \w0 \end{pmatrix}},                                     &        & p_{\mu,1}                        & = & +2,\\[0.1cm]
    \mathrm{a}_{\mu,2}     & = & \z{\begin{pmatrix} \w0 & \w0 &  \w{k_{\textrm{CO}}^{\textrm{De}}} \end{pmatrix}},                                    &        & p_{\mu,2}                        & = & -2, \\[0.1cm]
    \mathrm{a}_{\mu,\mu+1,1} & = & \begin{pmatrix} k_{\textrm{O}_2}^{\textrm{Ad}} & 0 & 0 \\ \w0 & \w0 & \w0 \\ 0 & 0 & 0 \end{pmatrix},    & \qquad & [ p_{\mu,\mu+1,1}, \, q_{\mu,\mu+1,1}] & = & [+1,\, +1],\\[0.5cm]
    \mathrm{a}_{\mu,\mu+1,2} & = & \begin{pmatrix} \w0 & \w0 & \w0 \\ 0 & k_{\textrm{O}_2}^{\textrm{De}} & 0  \\ 0 & 0 & 0 \end{pmatrix},   &        & [ p_{\mu,\mu+1,2}, \, q_{\mu,\mu+1,2}] & = & [-1,\, -1],\\[0.5cm]
    \mathrm{a}_{\mu,\mu+1,3} & = & \begin{pmatrix} \w0 & \w0 & \w0 \\ 0 & 0 & 0 \\ 0 & k_{\textrm{CO}_2}^{\textrm{De}} & 0\end{pmatrix},    &        & [ p_{\mu,\mu+1,3}, \, q_{\mu,\mu+1,3}] & = & [-2,\, -1],\\[0.5cm]
    \mathrm{a}_{\mu,\mu+1,4} & = & \begin{pmatrix} \w0 & \w0 & \w0 \\ 0 & 0 & k_{\textrm{CO}_2}^{\textrm{De}}  \\ 0 & 0 & 0 \end{pmatrix},  &        & [ p_{\mu,\mu+1,4}, \, q_{\mu,\mu+1,4}] & = & [-1,\, -2],\\[0.5cm]
    \mathrm{a}_{\mu,\mu+1,5} & = & \begin{pmatrix} \w0 & \w0 & \w0 \\ k_{\textrm{O}}^{\textrm{Diff}} & 0 & 0 \\ 0 & 0 & 0 \end{pmatrix},    &        & [ p_{\mu,\mu+1,5}, \, q_{\mu,\mu+1,5}] & = & [-1,\, +1],\\[0.5cm]
    \mathrm{a}_{\mu,\mu+1,6} & = & \begin{pmatrix} 0 & k_{\textrm{O}}^{\textrm{Diff}} & 0 \\ \w0 & \w0 & \w0 \\  0 & 0 & 0 \end{pmatrix},   &        & [ p_{\mu,\mu+1,6}, \, q_{\mu,\mu+1,6}] & = & [+1,\, -1],\\[0.5cm]
    \mathrm{a}_{\mu,\mu+1,7} & = & \begin{pmatrix} \w0 & \w0 & \w0 \\  0 & 0 & 0 \\ k_{\textrm{CO}}^{\textrm{Diff}} & 0 & 0 \end{pmatrix},  &        & [ p_{\mu,\mu+1,7}, \, q_{\mu,\mu+1,7}] & = & [-2,\, +2],\\[0.5cm]
    \mathrm{a}_{\mu,\mu+1,8} & = & \begin{pmatrix} 0 & 0 & k_{\textrm{CO}}^{\textrm{Diff}} \\ \w0 & \w0 & \w0 \\  0 & 0 & 0  \end{pmatrix}, &        & [ p_{\mu,\mu+1,8}, \, q_{\mu,\mu+1,8}] & = & [+2,\, -2],
\end{array}$
\end{center}
for $\mu=1, \dots , d$, where $\mathrm{a}_{d,d+1,\ell} = \mathrm{a}_{d,1,\ell}$ and $[ p_{d,d+1,\ell}, \, q_{d,d+1,\ell}] = [ p_{d,1,\ell}, \, q_{d,1,\ell}]$. The output of the algorithm is then an MME operator $\mathbf{A} \in \R^{(3 \times 3) \times \dots \times (3 \times 3)}$ in TT~format with ranks at most 20.

\subsection{Stationary distribution}\label{app: I-A^T proof}

All eigenvalues of the infinitesimal generator $\mathbf{A}$ (and of its transpose) have non-positive real part. In particular, any stationary distribution $\mathbf{p}$ of the Markov process fulfills $\mathbf{A}^{T}\mathbf{p} = 0$, see~\cite{Keizer1972}. To show that there exists a unique stationary distribution, it is sufficient to verify that the CO oxidation model is ergodic. That is, every state must be reachable from every other state. Since the state space is finite, ergodicity is equivalent to irreducibility. Thus, we must check whether every pair of states is connected by a reaction path with strictly positive transition rates. For this purpose, we consider all essential configurations of a 3-site system (unique up to symmetry and simple diffusion), listed in Table~\ref{tab:essential_conf}.

\begin{table}[htbp]
\centering
\begin{tabular}{|rlcrc|crlcr|} 
\hline
1. &  $(\varnothing$,& $\varnothing,$ & $\varnothing)$ & & & 7. &  $(\textrm{CO}$,& $\textrm{CO},$ & $\textrm{CO})$ \\
2. &  $(\textrm{O}$,& $\varnothing,$ & $\varnothing)$ & & & 8. &  $(\textrm{CO}$,& $\textrm{O},$ & $\varnothing)$ \\
3. &  $(\textrm{O}$,& $\textrm{O},$ & $\varnothing)$ & & & 9. &  $(\textrm{CO}$,& $\textrm{O},$ & $\textrm{O})$ \\
4. &  $(\textrm{O}$,& $\textrm{O},$ & $\textrm{O})$ & & & 10. &  $(\textrm{O}$,& $\textrm{CO},$ & $\textrm{O})$ \\
5. &  $(\textrm{CO}$,& $\varnothing,$ & $\varnothing)$ & & & 11. &  $(\textrm{CO}$,& $\textrm{CO},$ & $\textrm{O})$ \\
6. &  $(\textrm{CO}$,& $\textrm{CO},$ & $\varnothing)$ & & & 12. &  $(\textrm{CO}$,& $\textrm{O},$ & $\textrm{CO})$ \\
\hline
\end{tabular}
\caption{Essential configurations for a 3-site CO oxidation model}
\label{tab:essential_conf}
\end{table}
To prove irreducibility, it suffices to show that each of the 12 configurations can be transformed into the empty configuration and back via reaction paths with strictly positive rates. Case 1 is trivial. All configurations except 2, 4, 8, 10, 11, and 12 can already be transformed into the empty configuration and back by CO and O$_2$ adsorption/desorption. Configuration 10 requires an additional diffusion step. Configurations 8, 11, and 12 can be transformed into configuration 2 via CO desorption and diffusion. The corresponding paths are reversible. Configuration 4 can likewise be reversibly transformed into configuration 2 by O$_2$ desorption and adsorption. Hence, it remains only to consider Configuration 2. A representative reversible reaction cycle connecting configuration 2 to the empty configuration is
\begin{equation*}
(\textrm{O}, \varnothing, \varnothing) \rightarrow (\textrm{O}, \textrm{CO}, \varnothing) \rightarrow \underline{(\varnothing, \varnothing, \varnothing)} \rightarrow (\textrm{O}, \textrm{O}, \varnothing) \rightarrow (\textrm{O}, \textrm{O}, \textrm{CO}) \rightarrow (\textrm{O}, \varnothing, \varnothing)
\end{equation*}
This shows that every $3$-site system is irreducible provided all reaction rate constants are positive. For $d$-dimensional systems, the results can be directly extended by sequentially considering $3$-site subsystems to transform arbitrary configurations into one another. 

\subsection{Results of kMC simulations}\label{app: oxidation - kMC}

For the kMC reference data, we generated $128$ statistically independent kMC trajectories with $10^9$ steps. We then obtain estimates of the stationary expected values (and sampling variances) by time averaging. The first $10^7$ steps have been excluded to remove the bias from the initial relaxation to the steady state. The estimated turn-over frequencies and and their standard deviations can be found in Tables~\ref{table: TOFs} and \ref{table: STDs}.

\begingroup
\renewcommand{\arraystretch}{1.3}
\begin{table}[htbp]
\centering
\scalebox{0.85}{
\begin{tabular}{rccccc}
\hline
\textit{\textbf{d}}  & \textbf{CO adsorption}    & \textbf{CO desorption}    & \textbf{CO$_2$ desorption} & \textbf{O$_2$ adsorption} & \textbf{O$_2$ desorption} \\ 
\hline
4  & 3.5422696299e-05 & 1.5185997458e-05 & 2.0237687919e-05 & 1.0040245419e-03 & 9.9389922261e-04 \\
8  & 3.7328722424e-05 & 1.5823140365e-05 & 2.1505585525e-05 & 1.0042757000e-03 & 9.9351700468e-04 \\
12 & 3.7333084931e-05 & 1.5824382231e-05 & 2.1508265710e-05 & 1.0042638791e-03 & 9.9351611263e-04 \\
16 & 3.7334456714e-05 & 1.5825601790e-05 & 2.1509619046e-05 & 1.0042770629e-03 & 9.9351585060e-04 \\
20 & 3.7335184883e-05 & 1.5825615091e-05 & 2.1509901799e-05 & 1.0042688156e-03 & 9.9351569748e-04 \\
24 & 3.7334884155e-05 & 1.5824954159e-05 & 2.1509156606e-05 & 1.0042755941e-03 & 9.9351576584e-04 \\
28 & 3.7336604446e-05 & 1.5825553851e-05 & 2.1510439516e-05 & 1.0042775430e-03 & 9.9351541943e-04 \\
32 & 3.7333109201e-05 & 1.5825330505e-05 & 2.1509577933e-05 & 1.0042572308e-03 & 9.9351609823e-04 \\
\hline
\end{tabular}}
\caption{Turn-over frequencies of the adsorption and desorption reactions}
\label{table: TOFs}
\end{table}
\endgroup

\begingroup
\renewcommand{\arraystretch}{1.3}
\begin{table}[htbp]
\centering
\scalebox{0.85}{
\begin{tabular}{rccccc}
\hline
\textit{\textbf{d}}  & \textbf{CO adsorption}    & \textbf{CO desorption}    & \textbf{CO$_2$ desorption} & \textbf{O$_2$ adsorption} & \textbf{O$_2$ desorption} \\ 
\hline
4  & 8.8692752364e-10 & 5.9702663842e-10 & 8.3961750014e-10 & 5.7241856920e-09 & 1.7597058852e-10 \\
8  & 9.8599955756e-10 & 6.0868075621e-10 & 8.6929579961e-10 & 5.5141292819e-09 & 1.9680565609e-10 \\
12 & 1.1351048654e-09 & 6.2652115752e-10 & 8.9097244883e-10 & 5.3233453178e-09 & 2.2703866906e-10 \\
16 & 9.1150289747e-10 & 5.9237459877e-10 & 8.4782592763e-10 & 5.9227884377e-09 & 1.8085464858e-10 \\
20 & 9.9697286153e-10 & 5.8819891514e-10 & 8.8011499285e-10 & 5.7903015048e-09 & 1.9828456923e-10 \\
24 & 1.0546053696e-09 & 6.8317615168e-10 & 9.8665206189e-10 & 5.9304531817e-09 & 2.1075127101e-10 \\
28 & 9.5060163482e-10 & 6.0692586733e-10 & 8.5346603428e-10 & 5.8033600343e-09 & 1.8960981703e-10 \\
32 & 9.9071894806e-10 & 6.3079554592e-10 & 9.0198868970e-10 & 5.7879698723e-09 & 1.9754868488e-10 \\
\hline
\end{tabular}}
\caption{Standard deviations of the TOFs shown in Table \ref{table: TOFs}}
\label{table: STDs}
\end{table}
\endgroup

\section*{Acknowledgements}

We thank Marian Stengl for helpful discussions and insightful ideas that greatly enhanced the numerical experiments conducted in this study. The work of P.G.~was supported by the DFG Cluster of
Excellence MATH+ (EXC-2046/2, project ID 390685689). The work of A.U.~was supported by the Deutsche Forschungs\-gemeinschaft (DFG, German Research Foundation) – Projektnummer 506561557. 

{\small
\bibliographystyle{plain}
\bibliography{nearest}
}

\end{document}